\documentclass[12pt]{amsart}

\usepackage[margin=1.35in]{geometry}
\usepackage{amscd,amsmath,amsxtra,amsthm,amssymb,stmaryrd,xr,mathrsfs,mathtools,enumerate,commath, comment}
\usepackage{stmaryrd}
\usepackage[dvipsnames]{xcolor}
\usepackage{commath}
\usepackage{comment}
\usepackage{tikz-cd}
\usepackage{longtable} 
\usepackage{pdflscape} 
\usepackage{booktabs}
\usepackage{hyperref}
\hypersetup{
 colorlinks=true,
 linkcolor=Maroon,
 filecolor=Plum,      
 urlcolor=Purple,
 citecolor=RedViolet,
 pdftitle={Iwasawa Invariants for symmetric square representations},
    }
\usepackage{enumerate}
\usepackage[utf8]{inputenc}
\usepackage[OT2,T1]{fontenc}
\usepackage{color}

\DeclareSymbolFont{cyrletters}{OT2}{wncyr}{m}{n}
\DeclareMathSymbol{\Sha}{\mathalpha}{cyrletters}{"58}

\numberwithin{equation}{section}
 
\usepackage{color}
\DeclareSymbolFont{cyrletters}{OT2}{wncyr}{m}{n}
\DeclareMathSymbol{\Sha}{\mathalpha}{cyrletters}{"58}

\newcommand{\Qcyc}{\Q_{\op{cyc}}}
\newcommand{\Sel}{\op{Sel}_{p^\infty}}

\newcommand{\bfV}{\mathbf{V}}
\newcommand{\bfA}{\mathbf{A}}

\newcommand{\Z}{\mathbb{Z}}
\newcommand{\cO}{\mathcal{O}}
\newcommand{\cH}{\mathcal{H}}
\newcommand{\bfT}{\textbf{T}}

\newcommand{\p}{\mathfrak{p}}
\newcommand{\Q}{\mathbb{Q}}

\newcommand{\op}[1]{\operatorname{#1}}
\newcommand{\F}{\mathbb{F}}

\theoremstyle{plain}
 \theoremstyle{definition}
\newtheorem{Th}{Theorem}[section]
\newtheorem{Lemma}[Th]{Lemma}

\newtheorem{hypothesis}[Th]{Hypothesis}

\newtheorem{Corollary}[Th]{Corollary}
\newtheorem{Proposition}[Th]{Proposition}
\newtheorem{Remark}[Th]{Remark}
 \theoremstyle{definition}
\newtheorem{Definition}[Th]{Definition}
\newtheorem{Conjecture}[Th]{Conjecture}

\newcommand{\cc}{{\mathbf{C}}}

\newcommand{\zz}{{\mathbb{Z}}}

\newcommand{\T}{{\mathbf{T}}}

\newcommand{\qbar}{\bar{\Q}}

\newcommand{\frm}{{\mathfrak m}}

\newcommand{\gal}{{\operatorname{Gal}}}
\newcommand{\frob}{{\operatorname{Frob}}}

\renewcommand{\ker}{{\operatorname{ker}}}

\renewcommand{\hom}{{\operatorname{Hom}}}
\newcommand{\ord}{{\operatorname{ord}}}

\newcommand{\ol}{\overline}

\newcommand{\oo}{{\mathcal O}}

\newcommand{\pp}{{\mathfrak p}}

\newcommand{\sym}{{\operatorname{Sym^2}}}

\newcommand{\ad}{{\operatorname{Ad}}}

\begin{document}

\title[Iwasawa Invariants for Symmetric Square Representations]{Iwasawa Invariants for Symmetric Square Representations}

\author[A.~Ray]{Anwesh Ray}
\email[Ray]{anweshray@math.ubc.ca}
\author[R.~Sujatha]{R. Sujatha}
\email[Sujatha]{sujatha@math.ubc.ca}
\author[V.~Vatsal]{Vinayak Vatsal}
\email[Vatsal]{vatsal@math.ubc.ca}
\address{Department of Mathematics\\
University of British Columbia\\
Vancouver BC, Canada V6T 1Z2}

\subjclass[2010]{11R23 (primary)}
\keywords{Iwasawa theory, symmetric square representations, congruences, p-adic L-functions}

\begin{abstract}
Let $p\geq 5$ be a prime, and $\p$ a prime of $\overline\Q$
above $p$. Let $g_1$ and $g_2$ be $\p$-ordinary, $\p$-distinguished
and $p$-stabilized cuspidal newforms of nebentype characters $\epsilon_1, \epsilon_2$ respectively, 
and even weight $k\geq 2$, whose associated newforms have level prime to $p$. 
Assume that the residual representations at $\p$ associated to $g_1$ and $g_2$ are absolutely irreducible and isomorphic. Then, the imprimitive $p$-adic L-functions associated with the symmetric square representations are shown to exhibit a congruence modulo $\p$. Furthermore, the analytic and algebraic Iwasawa invariants associated to these representations of the $g_i$ are shown to be related. Along the way, we give a complete proof of the integrality of the $\p$-adic L-function, normalized with Hida's
canonical period. This fills a gap in the literature, since, despite the result being widely accepted, no 
complete proof seems to ever have been written down. On the algebraic side, we establish the 
corresponding congruence for Greenberg's Selmer groups, and verify that  the Iwasawa
main conjectures for the twisted symmetric square representations for $g_1$ and $g_2$ are compatible
with the congruences.
\end{abstract}

\maketitle

\section{Introduction}

\par Let $g_1$ and $g_2$ be eigencuspforms 
of the same even weight $k\geq 2$ and levels $M_1$ and 
$M_2$ respectively, normalized so that  $a(1, g_1)=1, i =1, 2$, where $a(1, g_i)$
denotes the $1$-st Fourier coefficient of $g_i$. 
Throughout, we fix a prime $p\geq 5$, and $\p|p$ a prime in the ring of integers $\cO_L$ 
of a suitably large number field $L$ containing the Hecke eigenvalues of $g_1$ and $g_2$,
as well as the values of the nebentype characters $\epsilon_i$ of the $g_i$. 
Let $\mathcal{O}$ denote the completion of $\cO_L$ at $\p$ and write $K$ for
the fraction field of $\cO$.  We
also fix an isomorphism of  $\cc_p$ with $\cc$, where $\cc_p$ is the completion of $\qbar$ at 
a prime above $\p$.
Assume 
that all but finitely many of the Hecke eigenvalues at primes $\ell$ of the $g_i$ are
congruent modulo $\p$. Then, we say that the newforms
$g_i$ are $\p$-congruent, or simply $p$-congruent, since $\p$ is fixed.
In the situation of $p$-congruent newforms, there is a general 
philosophy that the critical values of any L-functions functorially associated
to $g_1$ and $g_2$ will also be congruent, in a suitable sense. 
Furthermore, one 
expects the $p$-adic L-functions of $g_1$ and $g_2$, if they exist, to also be congruent. 
Finally, the corresponding $p$-primary Selmer groups defined over the cyclotomic 
$\Z_p$-extension of $\Q$ should also be related.

\par One case where these expectations can be verified was given by Greenberg and the third author of the present work
in \cite{GV00}, where they
studied the main conjecture of Iwasawa theory for 
the standard 2-dimensional representations associated to modular forms. In the 2-dimensional case, 
the $p$-adic L-functions of ordinary $p$-congruent modular forms
are the well-known $p$-adic L-functions arising from modular symbols; 
the Selmer groups are Greenberg's $p$-ordinary Selmer groups; and
the general expectation of the previous paragraph take the following precise shape:
the Iwasawa invariants of the Selmer groups and $p$-adic
L-functions, associated to congruent Hecke eigencuspforms, are related by an explicit formula.
The authors of \cite{GV00} combine their formula with deep results of Kato \cite{kat04}
to deduce certain cases of the Main Conjecture of Iwasawa
theory.

The primary goal of the present paper is to generalize the explicit relationship
between the Iwasawa invariants of the congruent forms $g_1, g_2$ from 
the case of the standard representation of dimension 2 to the case of
the symmetric square representation, which has dimension 3. A secondary 
accomplishment in this paper is the complete proof of the integrality  
for the $p$-adic L-functions of degree 3, since this result (although apparently known to experts)
seems not to be found in the literature. Implicit in this discussion
of integrality is a careful normalization of the periods appearing in the definition
of the $p$-adic L-function. This is a subtle point: it turns out to be
quite difficult to show that the normalization which gives rise to congruences
coincides with the canonical normalization given by Hida. We show that Hida's period
gives the correct congruences if a certain variant of Ihara's lemma holds. This lemma 
is unknown in weight $k >2$ in the generality which we require, although cases of it
are known weight 2. We remark also that all the results in this paper are much
easier to prove in the case of weight $2$ and the main novelty lies in the results for higher 
weight.

To state our results, we require some hypotheses and notation. We work
basically under the same hypotheses as in Loeffler-Zerbes \cite{lz16} and Schmidt \cite{schmidt88}, which we now describe.
Let $a(n, g_i)$  
denote the $n$-th Fourier coefficient 
of the eigenform $g_i$. In this paper, we make the following running assumptions. 
\begin{itemize}
    \item The forms $g_i$ are $\p$-ordinary, i.e., that $a(p, g_i)$ is a $\p$-adic unit for $i=1,2$;
    \item The $g_i$ are $p$-stabilized newforms, meaning that the level
$M_i$ of $g_i$ is divisible by precisely the first power of $p$; that each $g_i$ 
is an eigenvector for the $p$-th Hecke
operator $U_p$ (with unit eigenvalue at $\p$), and that  
the level of the newform $g_{0, i}$ 
associated to $g_i$ has level $M_i$ or $M_i/p$.
\item The nebentype characters $\epsilon_i$  are trivial on $(\Z/p\Z)^*$, and 
\item the  common 2-dimensional Galois representation with values in $\op{GL}_2(\bar{\F}_p)$ associated to the
    $g_i$ and $\p$ is  irreducible and $\p$-distinguished (in the sense of \cite{wil95}, page 481). 
    \end{itemize}

 \begin{Remark} The cases above cover situations where the representations $\rho_{g_i}$ are either
 crystalline or semistable. We have excluded the case of potentially good reduction, namely,
 ordinary eigenforms of level $M_0p^r,  (M_0, p)=1$,
 such  the character  $\epsilon_i$ of $g_i$ is nontrivial on
 $(\Z/p^r\Z)^*$ .  The exclusion is only to keep the calculations simple, 
 and the paper relatively short,
 as all the explicit
 formulae take on a different shape in the potentially good case. In particular, there are additional Euler-like factors in the
 interpolation formula for the $p$-adic
 L-function (see \cite{hid90}, equation (0.2) on page 97), 
 and the definition of the canonical integral period must also be adapted. Our method below applies to
 ordinary $p$-new forms, as well as to $\Lambda$-adic forms, but the calculations are repetitive, and 
 do not add anything new beyond different cases of the formulae. Note that Loeffler-Zerbes \cite{lz16} and Schmidt \cite{schmidt88}
 do not treat the case of ordinary semistable reduction. Our treatment applies to this case without change. 
 \end{Remark}
 
\begin{Remark} It is well-known that the symmetric square of a CM form is reducible, and that many statements
 about the symmetric square that hold in the non-CM case are no longer true. 
 However, we do \emph{not} exclude the case that the $g_i$ are CM forms. 
 We assume only that the Selmer groups defined below
 are cotorsion (which is known in the non-CM case \cite{lz16} under some hypotheses,  
 and also known in the CM case
 whenever it is expected to be true \cite{mw86}, \cite{rub91}) and that the auxiliary characters $\psi$ under consideration 
 are suitably restricted to avoid the problematic cases. The restriction
 on $\psi$ could be removed by following the analysis in \cite{hid90}, but we have not pursued this. 
\end{Remark}

To continue, we require more notation. Let the $g_i$ be as above.
Let $\Q_{\op{cyc}}$ denote the cyclotomic $\Z_p$-extension of $\Q$, i.e., the unique $\Z_p$-extension of $\Q$ contained in $\Q(\mu_{p^\infty})$,
where $\mu_{p^\infty}$ denotes the $p$-power roots of unity. Let 
$\Lambda=\cO[[\text{Gal}(\Q_{\op{cyc}}/\Q)]]$ denote the usual
Iwasawa algebra.  The group $\text{Gal}(\Q(\mu_{p^\infty})/\Q)$ is canonically identified
with $\zz_p^\times$ via the action on the $p$-power roots of unity.  
For a prime $q\neq p$, we normalize our Frobenius
elements so that $\frob(q)$ acts via $\zeta\mapsto \zeta^q$, where $\zeta$ is a $p$-power root of unity.
Then we may write $\frob(q)\in
\text{Gal}(\Q(\mu_{p^\infty})/\Q)\cong\zz_p^\times$ as $$\frob(q) = q = \eta_1(q)q_w$$
where $\eta_1$ is the Teichm\"uller character taking values in the $(p-1)$-th roots of unity, and $q_w$
is the projection of $q$
to the wild part, namely, projection
to the pro-$p$ subgroup $1+p\zz_p$. The reason for the subscript $1$ is to label
characters of $\zz_p^\times$; any character $\eta$ is a product of a wild part $\eta_w$, and a tame
part $\eta_t=\eta_1^t$. Eventually we will use this labelling of characters to write the interpolation formula
for the various power series giving rise to the $p$-adic L-functions.

If we fix a rank-$2$ Galois-stable $\cO$-lattice $T_{g_i}$ in the $p$-adic representation associated to $g_i$ 
(normalized so that $\frob(q)$ has 
trace equal to the $q$-Hecke eigenvalue of $g_i$, for almost all $q$)
we may view the representations $\rho_{g_i}$ as taking values in $\cO$. 
For $i=1,2$, let
\[\rho_{g_i}:\op{Gal}(\bar{\Q}/\Q)\rightarrow \op{GL}_2(\cO)\]be the associated Galois representation.
The residual representation is 
denoted by $\bar{\rho}_{g_i}:=\rho_{g_i}\mod{\p}$. 
Since the modular forms $g_1$ and $g_2$ are $\p$-congruent,
and since we have assumed that $\bar{\rho}_{g_1}$ and $\bar{\rho}_{g_2}$ are absolutely irreducible, 
we find that $\bar{\rho}_{g_1}$ 
and $\bar{\rho}_{g_2}$ are isomorphic. We will simply write $\ol\rho$ to denote the common residual 
representation. Since the residual representation is assumed to be irreducible, the lattices
$T_{g_i}$ are unique up to homothety.

Let $\psi$ denote a Dirichlet character of conductor $c_\psi$, where $(c_\psi, p)=1$. In this paper we will assume that
\begin{itemize}
    \item $\psi$ is even;
    \item the reduction of $(\psi\epsilon_i)^2\neq 1\mod{\pp}$, for the characters $\epsilon_i$ of the $g_i$; and
    \item the coefficient field $L$ contains the values of $\psi$.
\end{itemize}
We may view $\psi$ as a Galois character via $\psi(\frob(q))=\psi(q)$. Conversely, we may view a Galois
character as a Dirichlet character, simply by inverting this convention. In particular, 
the Teichm\"uller character $\eta_1$ mentioned above may be viewed as a Dirichlet character mod $p$, 
with values in $\mu_{p-1}$. The second condition above is inserted to rule out the problematic diehdral cases.

For each even integer $t$, we let $\psi_t$ denote the Dirichlet character $\psi\eta_1^t$. Let
$r_{g_i}=\op{Sym}^2(\rho_{g_i})$ denote the symmetric square representation
for $g_i$, with $i=1,2$, viewed as taking values in the symmetric square
of the lattice $T_{g_i}$. 
In this setting the representations
$r_{g_i}\otimes\psi:\op{Gal}(\bar{\Q}/\Q)\rightarrow \op{GL}_3(\cO)$ are residually isomorphic. Let $\bfA_{i, \psi_t}$ be the 
$p$-primary 
representation associated to the underlying Galois-stable $\cO$-lattice for $r_{g_i}\otimes\psi_t$. 
 Note that since $g_i$ is 
$p$-ordinary, 
so is $r_{g_i}\otimes\psi$. For each value of $t$ as above, 
we work with the primitive Selmer group $\op{Sel}_{p^\infty}(\bfA_{i,\psi_t}/\Q_{\op{cyc}})$ as defined by Greenberg in 
\cite{Gre89}. We are assuming that $\op{Sel}_{p^\infty}(\bfA_{i,\psi_t}/\Q_{\op{cyc}})$ 
is $\Lambda$-cotorsion,
 and this allows us to define a nonzero algebraic $p$-adic L-function 
$\mathscr{L}^\text{alg}(r_{g_i}\otimes\psi_t)\in \Lambda$ as a generator of the characteristic ideal of the Pontrjagin
dual of the Selmer group. It is well-defined up
to a unit factor in $\Lambda^\times$. We also point out that the existence of $\mathscr{L}^\text{alg}(r_{g_i}\otimes\psi_t)$ and part of the main conjecture is 
proven, in some cases, in unpublished
work of Urban \cite{urban06}. The arguments in this paper do not assume the main conjecture. 
By the Weierstrass preparation theorem, 
\[\mathscr{L}^\text{alg}(r_{g_i}\otimes\psi_t)=p^{\mu} a(T) u(T),\] where $a(T)$ is a distinguished polynomial and $u(T)$ is a unit in $\Lambda$ 
(each depending on $i$ and $t$). 
The $\mu$-invariant $\mu^{\op{alg}}(r_{g_i}\otimes\psi_t)$ is the number $\mu$ in the above factorization, and the $\lambda$-invariant $\lambda^{\op{alg}}(r_{g_i}\otimes\psi_t)$ is the degree of $a(T)$.

\par Next, we define a primitive $p$-adic L-function $\mathscr{L}^\text{an}(r_{g_i}\otimes\psi_t)$, for $i=1,2$.
This is essentially done in old work of Schmidt and others; see \cite{schmidt88} for the basic source,
and the discussion in \cite{lz16} for an account of the various refinements. 
In this paper, we shall normalize and label our $p$-adic L-functions as in \cite{Gre89}, to facilitate comparision with the Selmer groups defined there.

Under the present hypotheses, work of
Schmidt, Hida, and Dabrowski-Delbourgo proves the existence of an element 
$\mathscr{L}^\text{an}(r_{g_i}\otimes\psi_t)\in \Lambda\otimes\Q$ satisfying
a certain interpolation property with respect to special values of the complex symmetric square L-function. 
The interpolation property defining
the $p$-adic L-function
is given in formula (\ref{interpolation}) below. 

It is important to remark that the definition of  $\mathscr{L}^\text{an}(r_{g_i}\otimes\psi_t)$
presupposes the choice of a certain transcendental period.  We discuss our normalization 
and our choice of period in further detail in Section 2 of this paper, and summarize the key points later in this introduction
(see equation \eqref{canonical period def}).
 According to our normalization, the main conjecture predicts that the two power series
$\mathscr{L}^\text{alg}(r_{g_i}\otimes\psi_t)$ and $\mathscr{L}^\text{an}(r_{g_i}\otimes\psi_t)$
are equal up to a unit factor in $\Lambda$. 

Once the integrality of the $p$-adic L-function is established, we have 
well-defined Iwasawa invariants 
$\lambda^\text{an}(r_{g_i}\otimes\psi_t), \mu^\text{an}(r_{g_i}\otimes\psi_t)
\in \Z_{\geq 0}$ on the 
analytic side as well. The main conjecture implies that 
$$\lambda^\text{an}(r_{g_i}\otimes\psi_t)=\lambda^{\op{alg}}(r_{g_i}\otimes\psi_t)$$ and
$$\mu^\text{an}(r_{g_i}\otimes\psi_t)
=\mu^{\op{alg}}(r_{g_i}\otimes\psi_t)$$ for $i=1,2$. 

Our goal is to relate the invariants for the congruent forms $g_1$ and $g_2$. However, the primitive invariants
as defined above are not equal; it is quite possible for the primitive invariants to be trivial in one case 
yet highly nontrivial in the other. The correct relationship, as discovered in \cite{GV00}, is between
\emph{imprimitive} Iwasawa invariants, which we now proceed to define. Thus 
let $S_0$ denote any finite set of primes $q\neq p$. For each $i$,
we have an imprimitive Selmer group, obtained by relaxing the local
conditions at all the primes $q\in S_0$. It is shown
in \cite{GV00} that the imprimitive
Selmer group is cotorsion if and only if the primitive Selmer group
is so. Thus we have imprimitive invariants
$\lambda^{\op{alg}}_{S_0}(r_{g_i}\otimes\psi_t)$ and $\mu^{\op{alg}}_{S_0}(r_{g_i}\otimes\psi_t)$.  
For each prime $q$, we can define an integer $\sigma_i^{(q)}(\psi_t)$ to be the degree of a certain polynomial
coming from applying the Weierstrass preparation theorem to the 
annihilator of a local cohomology group at $q$, which is known
unconditionally to be torsion, and whose annihilator can be 
described explicitly in terms of Euler factors. Then the basic result is the
following.

\begin{Proposition}
\label{algebraic-invariants-intro} Assume that the Selmer groups $\op{Sel}_{p^\infty}(\bfA_{i,\psi_t}/\Q_{\op{cyc}})$
are cotorsion as $\Lambda$-modules.
The following statements hold:
\begin{itemize}
    \item $\mu_{S_0}^{\op{alg}}(r_{g_i}\otimes\psi_t)=
   \mu^{\op{alg}}(r_{g_i}\otimes\psi_t)$, and
  \item  $\lambda_{S_0}^{\op{alg}}(r_{g_i}\otimes\psi_t)=
   \lambda^{\op{alg}}(r_{g_i}\otimes\psi_t) + \sum_{q\in S_0} \sigma_i^{(q)}(\psi_t).$
\end{itemize}
\end{Proposition}

Most of this proposition is carried through from \cite{GV00}, 
where an analysis of the local conditions and cohomology groups
is made under very mild hypotheses. 

Similar considerations apply on the analytic side.
By virtue of a construction due to Coates, Hida, and Schmidt, there
exists an imprimitive $p$-adic L-function $\mathscr{L}^{\op{an}}_{S_0}(r_{g_i}\otimes\psi_t)\in \Lambda\otimes\Q$.
Once again, $S_0$ denotes a finite set of primes $q\neq p$. 

We want to study congruences and integrality, but the results of 
Hida, Coates, and Schmidt only produce elements of $\Lambda\otimes\Q$. Thus we have to show
that the $p$-adic L-functions actually lie in $\Lambda$. 
This integrality property  is a folklore result for which a complete proof seems 
not to have ever been written down. It is stated as Proposition 2.3.5
in \cite{lz16}, where it is attributed to Hida, but no reference is given. 
Despite searching many papers of Hida, we were unable to find a complete proof, except
for the sketch in \cite{hida-adjoint-L-values}, which is only complete in weight 2, and subject
to various complicated restrictions.  
Therefore, our first task is to provide a through discussion of integrality, and give an 
integral construction valid for all weights. The proof turns out to be somewhat delicate,
 once the weight is large compared to $p$. 
 And, as we have remarked above, the choice of periods is important
 and must be specified precisely. For now, we simply record that Hida's theory
 leads to a choice of period $\Omega_{g_i}^{\op{can}}$ that is determined
 up to a $\p$-adic unit, and that it is this choice that one would like to use. The precise definition
 of $\Omega_{g_i}^{\op{can}}$
 is given in equation \eqref{canonical period def}. 
  
To state the result, let us fix $g=g_i, M=M_i, \epsilon =\epsilon_i$. In this theorem, $\psi$ 
will denote an even character of conductor
$c_\psi$ coprime to $p$ such that $(\psi\epsilon_i)^2\neq 1\mod{\pp}$, and $S_0$ is any finite set of primes $q\neq p$.
 Recall
that we have set $\psi_t=\psi\eta_1^t$, for any even $t$. 

\begin{Th}
\label{integrality-thm-intro} 
Let the notation be as above.
Then the primitive L-function $\mathscr{L}^\text{an}(r_{g}\otimes\psi_t)\in \Lambda$ is integral, 
normalized with the integral period $\Omega_g^{\op{can}}$. 
The imprimitive $p$-adic L-function
$\mathscr{L}^\text{an}_{S_0}(r_{g}\otimes\psi_t) \in \Lambda$ is integral as well, for the same choice
of period. 
\end{Th}


With the integrality theorem in hand, 
our main result  in Section 2 is the following. As stated here, the result is dependent
on the validity of a certain variant of Ihara's lemma (see Hypothesis \ref{hypothesis ihara}). 
One can make an unconditional statement at the cost
of introducing a certain ambiguity in the choice of periods, but the statement is somewhat clumsy, so we avoid it. 
Let
 $\lambda_{S_0}^{\op{an}}(r_{g}\otimes\psi_t)$ and $\mu^{\op{an}}_{S_0}(r_{g}\otimes\psi_t)$
denote the Iwasawa invariants of $L^\text{an}_{S_0}(r_{g}\otimes\psi_t) $.

\begin{Proposition} Let the notation and assumptions be as above. Assume further 
that Hypothesis \ref{hypothesis ihara} holds. Then the following statements hold.
\label{analytic-invariants-intro}
\begin{itemize}
    \item $\mu_{S_0}^\text{an}(r_{g}\otimes\psi_t)=
   \mu^{\op{an}}(r_{g}\otimes\psi_t)$, and
  \item  $\lambda_{S_0}^{\op{an}}(r_{g}\otimes\psi_t)=
   \lambda^{\op{an}}(r_{g}\otimes\psi_t) + \displaystyle\sum_{q\in S_0} \sigma_i^{(q)}(\psi_t)$
\end{itemize}
\end{Proposition}

Here the integers $\sigma_i^{(q)}(\psi_t)$ are the \emph{same} as the ones occurring
in the algebraic case, since, the Euler factors in the algebraic
and analytic sides are exactly the same. This is a deep fact: 
the equality between the Galois-theoretic Euler factors 
in the algebraic case and the complex Euler factors in the analytic
L-function is the local Langlands correspondence for the 
3-dimensional representations $r_{g}\otimes\psi_t$; see
\cite{GJ78}, or \cite{schmidt88}. 

With regard to the Iwasawa invariants for congruent forms, our result is as follows. 
Again, we include the Ihara Lemma
as a hypothesis, since formulating a general result without it would give a clumsy statement. We remark
however that the Ihara hypothesis may be proved to be true in weight 2. We will say more about this point
 in Section \ref{ihara-comments}. Consider the pair
$g_1, g_2$ of $p$-congruent forms, of level $M_1, M_2$ respectively. In the situation where we consider the
congruences, we have to delete the Euler factors at \emph{all} bad primes, just as in \cite{GV00}. We also
have to take care of the prime $q=2$, to apply the results of Shimura.
Thus, we let $S$ denote any set of primes $q$ containing
$q=2$ and all primes dividing $M_1M_2$. Let $S_0=S\backslash \{p\}$. 
\begin{Th}
\label{intro-thm}
Let the notation be as above. Assume Hypothesis \ref{hypothesis ihara}, 
and that the Selmer groups $\op{Sel}_{p^\infty}(\bfA_{i,\psi_t}/\Q_{\op{cyc}})$
are cotorsion as $\Lambda$-modules.
Then the following statements hold.
\begin{enumerate}
\item If $\mu_{S_0}^\text{an}(r_{g_1}\otimes\psi_t) = 0$, we have
$\mu_{S_0}^\text{an}(r_{g_2}\otimes\psi_t) = 0$, and 
$\lambda_{S_0}^\text{an}(r_{g_1}\otimes\psi_t) =
\lambda_{S_0}^\text{an}(r_{g_2}\otimes\psi_t)$.
\item If $\mu_{S_0}^{\op{alg}}(r_{g_1}\otimes\psi_t) = 0$, we have
$\mu_{S_0}^{\op{alg}}(r_{g_2}\otimes\psi_t) = 0$, and 
$\lambda_{S_0}^{\op{alg}}(r_{g_1}\otimes\psi_t) =
\lambda_{S_0}^{\op{alg}}(r_{g_2}\otimes\psi_t)$.
\item If $\mu_{S_0}^\text{an}(r_{g_1}\otimes\psi_t) =  \mu_{S_0}^{\op{alg}}(r_{g_1}\otimes\psi_t) = 0$, and
$\lambda_{S_0}^\text{an}(r_{g_1}\otimes\psi_t)=
\lambda_{S_0}^{\op{alg}}(r_{g_1}\otimes\psi_t)$, then
$\mu_{S_0}^\text{an}(r_{g_2}\otimes\psi_t) =  \mu_{S_0}^{\op{alg}}(r_{g_2}\otimes\psi_t) = 0$, and
$\lambda_{S_0}^\text{an}(r_{g_2}\otimes\psi_t)=
\lambda_{S_0}^{\op{alg}}(r_{g_2}\otimes\psi_t)$.
\end{enumerate}
\end{Th}

It is clear from the relationships given in Propositions \ref{algebraic-invariants-intro} and 
\ref{analytic-invariants-intro}, that the third statement follows from
the first two. Furthermore, it follows from the third statement
that if one knows the main conjecture and vanishing of the $\mu$-invariants for $g_1$,
that the same conclusions follow for $g_2$. Examples where the main conjecture is known for a particular
form may be found in \cite{lz16}. 

\begin{Remark} Our theorems come in two flavours -- some  (e.g. Theorem
\ref{integrality-thm-intro}, and Proposition \ref{analytic-invariants-intro}) pertain 
to properties of a single form $g=g_i$, whilst others (Theorem \ref{intro-thm}) pertain to two congruent forms $g_1$ and $g_2$.
In the theorems dealing with a single form, the set $S_0$ is allowed to be an arbitrary set of
primes $q\neq p$, while in the case of congruences, we must assume that $S_0$ contains all primes
dividing the levels $M_1, M_2$, together with the prime $2$. 
\end{Remark}

In the rest of this introduction, we give a sketch of the arguments, indicating the various difficulties,
and the novelties needed to overcome them. 
We also take the opportunity to introduce some of the quantities and
notation that we will use in the course of this paper.

The main difficulties occur on the analytic side, starting with the integrality property described above.
The proof of the congruences of $p$-adic L-functions turns out to be 
 more delicate than in the case of \cite{GV00}, and relies crucially on knowledge of the possible
non-minimal deformations of $\ol\rho$ (Lemma \ref{gouvea90} below).
Indeed, to get the required results, one has to redo the Coates-Hida-Schmidt construction of the $p$-adic 
L-function, which goes back more than 30 years, 
and apply various refinements which were not available at that time. 

We start by writing down the Euler products and L-functions that are the subject of this paper.
The story is a bit complicated, because our \emph{results} concern the primitive and imprimitive automorphic
L-functions of $GL_3$ but the \emph{proofs} are concerned almost exclusively with certain related
degree 3 Euler products defined by Shimura, which are almost -- but not quite -- the same. Since the distinction
is important in understanding why the objects in the sketch of the proofs do not appear in the statements
of the theorems, 
we spell out the definitions in this introduction.

For notational simplicity, assume that $g=g_i$ is a $p$-stabilized
newform for some 
fixed value $i\in \{1,2\}$ and set $M:=M_i$. Then $M$ is divisible by precisely the first power of $p$. 
If  $g=\sum a(n ,g)q^n$, then the Dirichlet series $\sum_n a(n, g)n^{-s}$ has an Euler product of the form
$$\prod_q(1-\alpha_qq^{-s})^{-1}(1-\beta_qq^{-s})^{-1}$$
with certain parameters $\alpha_q, \beta_q$ at each prime $q$ (including $p$). If $q$ divides $M$, then
 one or both of these parameters may be zero. Whether or not these factors are zero may be succinctly 
 described as follows.
 
 Let $T\cong \oo\oplus\oo $ denote a realization of the Galois representation $\rho_g=\rho_{g_i}$, 
 as in the previous discussion above. If $q\neq p$ is a prime, we shall say that $\rho_g$ is \emph{ordinary} at $q$
 if the submodule $T^{I_q}$ of invariants under an inertia group $I_q$ has $\oo$-rank $1$. This terminology
 is taken from \cite{gouvea90}. Analogously, we say that $q$ is unramified if $T^{I_q}=T$ has rank 2, 
 and that $T$ is depleted if $T^{I_q}=0$. Then it is known that precisely one of $\alpha_q, \beta_q$ is non-zero
 when $q$ is ordinary; that both are zero if $q$ is depleted; and that both are non-zero when $q$ is unramified.
 As for $q=p$, we are in the $p$-stabilized case, so $\alpha_p$ is a unit, and $\beta_p=0$. 
 
Now let $\eta$ be an even finite order Dirichlet character of conductor $c_\eta=p^r$, and set
$\chi=\psi \eta$. Thus $\chi$ is an even Dirichlet character of conductor $c_\psi p^r$. 
We may write 
$\eta=\eta_t\eta_w$, where $\eta_t =\eta_1^t$, for the Teichm\"uller character $\eta_1$, and a character
$\eta_w$ of $p$-power order, so that
$\chi =\psi_t\eta_w$. We will use this notation throughout the paper.

 Let $S_0$ denote any set of prime numbers such that $p\notin S_0$. We allow that $S_0$ could be empty. 
Then the $S_0$-imprimitive symmetric square $L$-function (twisted by $\chi$)   is defined as follows:

\begin{equation}
\label{shimura-euler-product}
{L}_{S_0}(r_g\otimes\chi, s) = 
\prod_{q\notin S_0} P_q(q^{-s})^{-1}
\end{equation}
where $P_q(X)$ is a certain polynomial with algebraic integer coefficients (see Section 
\ref{primitive-ss-Lfunction}). There is no simple
description of the polynomials $P_q$, beyond the fact  that $P_q(q^{-s}) =
(1-\chi(q)\alpha_q\beta_q q^{-s})(1-\chi(q)\beta_q^2q^{-s})(1-\chi(q)\alpha_q^2q^{-s})$
for almost
all $q$, where $\alpha_q$ and $\beta_q$ are determined from the standard degree 2 Euler product given above. 
The exact form of the finitely many remaining factors $P_q$
may be found in \cite{schmidt88}, Section 1; we do not need the prescription here. The primitive
symmetric square L-function ${L}(r_g\otimes\chi, s)$ corresponds to the case $S=\emptyset$; it
is the L-function corresponding to an automorphic representation of $\operatorname{GL}_3$. 

Next we want to define Shimura's Euler products for the modular form $g$ and a certain depletion 
$f$ of $g$. Given the $p$-stabilized
modular form $g$ of level $M$, define
$$D_g(\chi, s) = \prod_q \left( (1-\chi(q)\alpha_q\beta_q q^{-s})(1-\chi(q)\beta_q^2q^{-s})(1-\chi(q)\alpha_q^2q^{-s})\right)^{-1},$$
where the product is taken over all primes $q$. This function is written as $D(s, g, \chi)$ by Shimura in \cite{shimura-holo}. We note
that Shimura's results don't require that $g$ be a newform; his results apply to any cuspform that is an eigenform for all the Hecke 
operators. It is not necessarily true that $D_g(\chi, s)$ coincides with the automorphic L-function $L(r_g\otimes\chi, s)$. 

The analogue of the imprimitive L-function $L_{S_0}(r_g\otimes\chi, s)$ is given by $D_f(\chi, s)$, where $f$ is the depletion of 
$g$ at the set of primes in $S_0$, defined as follows. If $g =\sum_n a(n, g)e^{2\pi i nz}$, then $f=\sum_n' a(n, g)e^{2\pi i nz}$ where the second 
sum is taken over integers $n$ such that $(n, q)=1$ for all  $q\in S_0$. The modular form $f$ is an eigenform of some level $N$,
and 
the standard $L$-function of $f$ admits a degree 2 Euler product. We get parameters $\alpha_q', \beta_q'$ associated 
to $f$ just as before,
so that $\alpha_q=\alpha'_q, \beta_q=\beta'_q$ if $q\notin S_0$, and $\alpha_q=\beta_q=0$ if not. 
Then we can follow Shimura and 
define a degree three Euler product $D_f(\chi, s)$ for $f$ just as above, with $\alpha_q',\beta_q'$ instead of $\alpha_q, 
\beta_q$. Again, it is not necessarily true in general that  $D_f(\chi, s) =L_{S_0}(r_f\otimes\chi, s)$. 
This fact does become true when $S_0$ is sufficiently large and (since $p\notin S_0$) when $\chi$ is ramified at $p$.  
For example, it suffices to assume that $S_0$ 
includes all the primes $q\neq p$ dividing the level $M$ of $g$
and the conductor $c_\chi$ of $\chi$. 


With these preliminaries on complex L-functions in hand, we can now describe the $p$-adic L-functions we consider. 
Let $g$ be a $p$-stabilized newform as above; let $S_0$ denote any sufficiently large finite set of primes $q\neq p$, and let
$f$ denote the depletion of $g$ at $S_0$. 
Let $\psi$ be a fixed even character of conductor relatively prime to $p$.
Set $\chi=\psi\eta=\psi_t\eta_w$, for $\eta$ of $p$-power conductor. Then
the imprimitive 
$p$-adic L-function $\mathscr{L}^\text{an}_{S_0}(r_{g}\otimes\psi_t)$ 
is defined via interpolation
of  $ D_f(\psi_t\eta_w, s)$ 
at critical values of $s$, and  $\eta_w$ varying over cyclotomic characters of $p$-power
order and $S_0$ is chosen sufficiently large to ensure that $D_f(\chi, s) =L_{S_0}(r_g\otimes\chi, s)$.
To write an interpolation formula, we view $\mathscr{L}^\text{an}_{S_0}(r_{g}\otimes\psi_t)$ as an element of
$\Lambda=\cO[[\text{Gal}(\Q_{\op{cyc}}/\Q)]]$. Given a finite order character $\eta_w$ of $\text{Gal}(\Q_{\op{cyc}}/\Q)\cong 1+ p\zz_p$, and a rational
integer $n$, let us write
$\eta_{w, n}$ for the character of $\text{Gal}(\Q_{\op{cyc}}/\Q)$ defined by
$$\eta_{w, n}(u) =\eta_w(u)\cdot u^{1-n}.$$ 
Thus we have  
$$\eta_{w, n}(\frob(q)) =(\eta_w\eta_1^{n-1})(\frob(q))q^{1-n},$$
for all primes $q\neq p$. 
 
The $p$-adic L-function is characterized (assuming it exists) by its specialization at any
infinite set of characters. We consider characters
of the form $\eta_{w, n}$, where $\eta_w$ is non-trivial, hence of conductor $p^r , r > 1$, and odd $n$ in the range
$1\leq n < k$. Let $$E_p(n,\eta_w) = E_p(n,\psi, \eta_w) =  (p^{n-1}\psi(p)^{-1}\alpha_p^{-2})^r.$$
Here $\alpha_p$ denotes the $U_p$ eigenvalue of $g$, as always. The number $\psi(p)$ is nonzero
because $\psi$ is assumed to be unramified at $p$. 
  If $\chi$ is any Dirichlet character, we let $G(\chi)$ denote the usual Gauss sum of $\chi$. 
  Fix an even $t$ with $0\leq t\leq p-2$. 
Then the characterizing interpolation formula is stated as follows: there exists a complex period 
$\Omega_{g_0}^{\op{can}}$ associated to the newform $g_0$, determined up to a $\p$-adic unit, and independent of  $S_0$ and the
corresponding depletion $f$ of $g_0$, such that
\begin{equation}
\label{interpolation}
\eta_{w, n}(\mathscr{L}^\text{an}_{S_0}(r_{g}\otimes\psi_t)) =  \Gamma(n) E_p(n, \eta_w)\cdot G(\eta_{1-n-t}\eta_w^{-1})
\frac{L_{S_0}(\psi_{t+n-1}\eta_w, n)}{\pi^n\Omega_g^{\op{can}}}.
\end{equation}
\begin{Remark} This formula is adapted from \cite{lz16}, Theorem 2.3.2, which deals with the primitive L-function 
in the case of good reduction.
We have
transposed the formula to the imprimitive and semistable 
case (the formula is the same), and suppressed various factors of $2$ and $i$, which depend only on $k$ 
and not on $n$. We have further simplified the formula by virtue of our assumptions that $n$ is odd and the character $\psi_t$
even.  We have not
given the interpolation formulae for the case when $\eta$  is trivial, as we do not need them. 
Our formula is written in terms of the canonical period, rather than the Petersson norm which appears in \cite{lz16};
the relationship between the two periods is given in equation \eqref{canonical period def} below.  Furthermore, 
we have adjusted the action of the Iwasawa algebra 
to match the conventions of \cite{Gre89}, since we are working with Greenberg's Selmer groups.  


We remark also 
that the congruence ideal which appears in the statement of the Main
Conjecture in  Section 2 of \cite{lz16} does not play a role for us, since our definition of the period will subsume this factor. 
Finally, we point out that in the convention of \cite{lz16}, $\Lambda$ refers to the completed
group algebra of $\gal(\Q(\mu_{p^\infty})/\Q)\cong\Z_p^\times$, whereas we have taken $\Lambda$ to be the group algebra
of $\gal(\Q_{\op{cyc}}/\Q)\cong \Z_p\cong 1+p\Z_p$, since the latter formulation is more convenient for
congruences. However, the case of nontrivial even tame cyclotomic twists is incorporated
by choosing a tame twist $\psi_t$, for each even $t$.

The term ``imprimitive'' is used in \cite{lz16} to refer 
to $D_g(\chi, s)$.  Our $D_f(\chi, s)$ 
is even less primitive than  the already defective
 $D_g(\chi, s)$. The function $D_f(\chi, s)$ does not appear in \cite{lz16}. 
\end{Remark}

For the present, we fix $S_0$ sufficiently large, and focus on $L_{S_0}(r_g\otimes\chi, s) = D_f(\chi, s)$, for $\chi=\psi\eta =\psi_t\eta_w$, 
with $\psi, \eta$ even, $\eta$ of $p$-power conductor, and $\psi$ of conductor prime to $p$,
and describe how to prove congruences for the quantity appearing on the right hand side of (\ref{interpolation}). 
The starting point is Shimura's formula expressing 
 $D_f(\chi, s)$ in terms of the Petersson inner product of a theta series $\theta$ and an Eisenstein series $\Phi^\epsilon$:
   \begin{equation}
  \label{shimura-formula-intro}
  (4\pi)^{-s/2}\Gamma(s/2) D_f(\chi, s) = \int_{B_0(N_\chi)} f(z), \overline{\theta_{\overline\chi}(z)\Phi^\epsilon(z,\overline\chi, s) } y^{k-2} dx dy
  \end{equation}
 where $z=x+iy$ runs over a fundamental domain $B_0(N_\chi)$ for the group $\Gamma_0(N_\chi)$. Here $N_\chi$ is the least
 common multiple of $4$, $N$, and the conductor of $\chi$. The superscript of $\epsilon$ on the Eisenstein series indicates a dependence
 on the character $\epsilon$ of the cusp form $f$.   If $n$ is an odd integer in the range
$1\leq n\leq k-1$,
Shimura in \cite{shimura-holo} has shown that $\theta_{\overline\chi}(z)\Phi^\epsilon(z,\overline\chi, n)$ is a nearly holomorphic modular form of level
$N_\chi$, weight $k$, and character $\epsilon$.  For the precise definitions of the theta function and Eisenstein series, see Section \ref{shimura formula} below. 

 \begin{Remark}
The integral in Shimura's formula is taken over a fundamental domain for $\Gamma_0(N_\chi)$.
The modular forms $f$ and $\theta_{\overline\chi}(z)\Phi(z,\overline\chi, s)$ both have nontrivial nebentype character,
but the integrand as a whole is invariant under $\Gamma_0(N_\chi)$. 
  \end{Remark}

 Our first job is to prove that the values on the right hand side of Shimura's formula
are algebraic and integral, when divided by the canonical period. The algebraicity is well-known, and was basically 
proven by Shimura himself:
it follows from the fact that the Fourier coefficients of $\theta_{\overline\chi}(z)\Phi^\epsilon(z,\overline\chi, n)$
can be calculated explicitly, and turn out to be algebraic. Then one 
has to replace the nearly holomorphic form $\theta_{\overline\chi}(z)\Phi^\epsilon(z,\overline\chi, n)$ with its holomorphic projection
$\theta_{\overline\chi}(z)\Phi^\epsilon(z,\overline\chi, n)^\text{hol}$. This projection will
have algebraic Fourier coefficients, so Shimura's method shows
that the Petersson product $\langle f, \theta_{\overline\chi}(z)\Phi^\epsilon(z,\overline\chi, n)^\text{hol}\rangle$ is equal to the
the Petersson inner product of $f$ with itself, up to an algebraic factor. 
This procedure is elaborated by Schmidt in \cite{schmidt86} and \cite{schmidt88},  in the case that $S_0$ is the empty set, 
so that $f=g$ and $N=M$. His work leads to the construction of the $p$-adic 
L-functions described above. The period that emerges is simply the Petersson inner product of  the $p$-stabilized newform 
$g$ with itself. 

Once one has a $p$-adic L-function for $S_0=\emptyset$, it is a straightforward matter to obtain
the $p$-adic L-function for general $S_0$: one has only to multiply by the finitely many 
Euler factors. However, this approach is useless for studying congruences, since one has no way to 
compare the constructions for 
forms of different minimal levels. One could simply multiply each one by an arbitrary unit factor, and 
the resulting objects
would not be congruent. Thus one has to construct both the L- functions simultaneously, 
at some fixed common level where 
suitably depleted forms co-exist. 
In other words, one has to verify that Schmidt's construction works for $D_f(\chi, s)$, when $S_0\neq\emptyset$;  this is
carried out in Section 2 below.  Here, the problem is that we have to work with the depleted form $f$, rather than the stabilized newform
$g$, so multiplicity one is not automatic.
Furthermore, we have to make sure our construction works over an integer ring, rather over a field. 

We describe the solution to the integrality problem first. 
The main difficulty is
that the holomorphic projection in Shimura and Schmidt introduces
denominators dividing $k!$. To get around this, we have to change tactics, and use methods from $p$-adic
modular forms -- we replace the holomorphic projection with the ordinary projection, which is denominator-free. 
This is enough for our purposes, since we are dealing with $f$ ordinary. 

Furthermore, one has to give an integral definition for the periods, rathen than just the Petersson product. 
The key idea (due to Hida) is that the Petersson inner product is related to a certain algebraic, and even integral, inner
product, up to canonical scalar multiple. Let  $M$ denote the level of the $p$-stabilized newform $g$.
let $S_k(M,\cO)$ be the space of cusp forms of weight $k$ on the group $\Gamma_1(M)$, with coefficients in $\cO$, 
and let $\bfT$ be the ring generated by Hecke operators acting on $S_k(M,\cO)$. Let $\mathcal{P}_f$ be the kernel 
of the map $\bfT\rightarrow \cO$ associated to $f$ and $\mathfrak{m}$ the unique maximal ideal of $\bfT$ 
generated by $\mathcal{P}_f$ and $\p$. We have assumed that the residual representation associated to $f$ is 
absolutely irreducible, ordinary, and $p$-distinguished, so it follows that $\bfT_{\mathfrak{m}}$ is Gorenstein. 
This induces an algebraic and integral duality pairing 
\[
(\;\cdot, \cdot)_N: S_k(M, \cO)_\mathfrak{m}\times S_k(M,\cO)_\mathfrak{m} \rightarrow \cO,\]
see \eqref{integral-pairing} and the discussion preceding it for more details.
Following Hida, we
 use the algebraic pairing defined to replace the usual Petersson inner product. To compare the two,
 we define a modified 
 Petersson product on $S_k(M, \mathbf{C})$ by setting
\begin{equation}
\label{modified-petersson-2}
\{v, w\}_{M} = \langle v\vert W_{M}, w^c\rangle_{M}
\end{equation}
where the pairing on the right is the Petersson product. The superscript $c$ denotes complex conjugation
on the Fourier coefficients, and $W_{M}$ is the Atkin-Lehner-Li involution. It is then shown that the two pairings are essentially scalar 
multiples of each other, thus $\{g , g\}_{M}=\Omega_M (g, g)_M$, where  $\Omega_M = \frac{\{g , g\}_M}{(g, g)_M}$ is well-defined
up to unit factor.  The number
$(g, g)_M$ is the so-called congruence number for $f$. We shall compute the number $\Omega_M$ more explicitly in Section 2 below, and
the computation leads to the following definition of the canonical period: we put
\begin{equation}
\label{canonical period def}
\Omega_{g_0}^{\op{can}}  = \frac{\langle g_0, g_0\rangle_{M_0}}{(g_0, g_0)_{M_0}} = \op{unit}\cdot p^{k/2-1}\cdot  \Omega_M
\end{equation}
where $g_0$ is the newform of level $M_0$ associated to $g$, so that $M_0= M$ or $M/p$, and 
$$\langle g_0, g_0\rangle_{M_0}= \int_{B_1(M_0)} |g_0(z)|^2y^{k-2} dxdy,$$
for $z= x+iy$ running over a fundamental domain $B_1(M_0)$ for $\Gamma_1(M_0)$. The number $(g_0, g_0)_{M_0}$ is a certain
congruence number associated to the newform $g_0$. The period $\Omega_g^{\op{can}}$ is defined up
to a $\p$-adic unit factor, coming from the choice of the Gorenstein isomorphism, which is fixed once and for all.

It remains to explain how to modify Schmidt's construction
to construct $p$-adic L-functions from depleted forms rather than newforms.
Some care is required, since Schmidt
relies crucially on some kind of semi-simplicity and multiplicity one result, in order to compute certain Petersson products which arise.  
In the depleted case, the Hecke algebra is not semi-simple, and it is unclear how to compute the Petersson products,
or even to show that they are nonzero. We are able to resolve
this problem because depletion at $q$ increases the level by at most $q^2$, and in fact the level doesn't increase at all unless $g$ is ordinary or unramified. 
This fact allows us to restrict attention to a 1-dimensional space cut out by Hecke operators.
 For example, if we deplete at an unramified prime $q\nmid M$, where $M$ is the level
of the $p$-stabilized newform $g$, then $U_q$  at level $N=Mq^2$ can have at most 3 different
eigenvalues: $\alpha_q, \beta_q, 0$, with $\alpha_q\beta_q =\epsilon(q)q^{k-1}$,  on any form obtained from $f$ by degeneracy maps.
Thus there is a unique $g$-isotypic line at level $Mq^2$ where $U_q=0$, and the depleted form $f$ is chosen to lie in exactly this 1-dimensional space. 

To deal with congruences, we actually need a bit more.
Indeed, for the purposes of relating the Iwasawa invariants of the $p$-congruent forms $g_1$ and $g_2$, we must
 show that we can simultaneously add primes to the level to obtain 
 imprimitive modular forms $f_1$ and $f_2$ of the \emph{same} level for which all Fourier coefficients are 
 $\p$-congruent, and for which all $U_q$ eigenvalues are zero, and for which we can retain some version
 of semisimplicity for the Hecke action. 
 The fact that the level so obtained 
 satisfies all our conditions is an application of a result due to Goueva,
(Lemma \ref{gouvea90} below). Once we have semisimplicity, and a perfect  integral pairing is shown to exist at a fixed
common level, it is a simple matter to show
 that normalized special values of $p$-adic $L$-functions associated to  $r_{g_1}\otimes\psi_t$ and $r_{g_2}\otimes\psi_t$ 
 are $\p$-congruent, see Theorem \ref{special values congruence}. The hard part of the theorem
 is showing that a suitable algebraic pairing exists at some common choice of level, and for this Gouvea's result is indispensable.

Finally, we mention a further -- and more stubborn -- point that arises in dealing with imprimitive forms. One would
like to compare the periods of the depleted form $f$ at level $N$ 
with the canonical periods of the newform $g_0$ at level $M_0$, since it is the imprimitive forms and imprimitive periods that show
up when dealing with congruences. 
Relating these periods
requires us to assume that a certain version of Ihara's lemma is satisfied, see Hypothesis 
\ref{hypothesis ihara}. Cases of the result are known unconditionally when $k=2$, but it is not yet resolved in all
the higher weight cases. We are therefore required to carry around  Hypothesis \ref{hypothesis ihara}.
If the reader is willing to allow the periods to depend on the level, then our results become unconditional.
We mention also that the results in this paper could be generalized to cover the remaining critical values, and odd
characters $\psi$, at the cost of additional calculation of Fourier expansions of theta functions, sometimes of weight $3/2$
and of slightly different Eisenstein series.

\par 
On the algebraic side, there is little difficulty. The results on Galois cohomology
in \cite{GV00} are quite general,
and apply to the situation treated here, so require little more than translation. 
The key condition is that the common reduction modulo $\p$ of the representations $r_{g_i}\otimes\psi_t$ 
 does not contain any non-zero  submodule on which the action of $\gal(\qbar/\Q)$ is trivial. This hypothesis
 is ensured by the condition that $(\psi\epsilon_i)^2\neq 1\mod{\pp}$.

In conclusion, we mention that the methods of this paper have been used by Delbourgo to prove congruences for Rankin-Selberg L-functions. We refer the reader to his forthcoming article for an interesting application.

\section*{Acknowledgments}
The authors would like to thank Haruzo Hida, Antonio Lei, Giovanni Rosso and Eric Urban for helpful comments. 
We are also grateful to the referees of this article for numerous insightful suggestions, and for the reference to \cite{gouvea90}.
The second  author gratefully acknowledges support from NSERC Discovery grant 2019-03987. The third  author was supported by the NSERC Discovery grant 2019-03929.

\section{Congruences for symmetric square L-functions}

\subsection{Definitions and normalizations}
\label{assumptions-and-definitions} 
\subsubsection*{Petersson inner product} Consider modular forms $v, w$ of weight $k\geq 2$
on $\Gamma_1(m)$ (at least one of which is cuspidal). Then the Petersson inner product
of $v$ and $w$ is defined via 
\begin{equation}
\label{standard-petersson}
\langle v, w\rangle_m = \int_{B_1(m)}v(z)\overline{w(z)} y^{k-2} dx dy,
\end{equation} 
where $B_1(m)$ is a fundamental domain for $\Gamma_1(m)$.
Since we will be comparing to the work of Schmidt and Shimura, it is important to record that their normalizations are
different: they write the Petersson product
 as an integral of $v\overline{w}$ on a fundamental domain for $\Gamma_0(m)$.
 
\subsubsection*{Slash operator and Hecke operators}
We follow the classical
conventions of \cite{miy89}, Chapter 4, 
when speaking of Hecke operators acting on modular forms. We refer also to \cite{ali}.
Let $m$ be any positive integer, which may be divisible by our fixed prime $p$. 
For  a prime $q\nmid m$ (resp. $q\mid m$, including possibly $q=p$) the Hecke operator $q^{1-k/2}T_q$ (resp. $q^{1-k/2}U_q$, ) of level $m$ corresponds
to the \emph{right} action of the double coset 
$\Gamma_1(m)\begin{pmatrix} 1 & 0 \\ 0 & q\end{pmatrix}\Gamma_1(m)$, acting via the slash operator,
normalized so that $(f\vert_k\gamma)(z) =\det(\gamma)^{k/2}(cz+d)^{-k}f((az+b)/cz+d)), \gamma =\begin{pmatrix}a & b \\ c& d\end{pmatrix}$.

The Hecke operators $T_q$ are normal with respect to the Petersson inner-product on $\Gamma_1(m)$, but the operators
$U_q$ are not. Write $W_m$ for
the operator on modular forms of level $m$ induced
by the action of the matrix $\begin{pmatrix} 0 & 1 \\  -m & 0\end{pmatrix}$.  Note that the matrix $W_m$
normalizes $\Gamma_1(m)$, and recall that $W_m^2$ acts via the scalar $(-1)^k$. Furthermore, the
adjoint of $U_q$ acting on cuspforms of level $m$ is the operator $U_q^*$ given by $W_m^{-1}U_q W_m$. 
Observe that the adjoint of $U_q$ depends on the level, although $U_q$ itself does not. 

Since we are working on 
$\Gamma_1(m)$, we must also specify additional operators giving the character: if $(n, m)=1$, then we have an operator
$S_n$ defined by $f\vert S_n = f\vert_k\sigma_n$, and $\sigma_n\equiv \begin{pmatrix} n^{-1}& *\\ 0& n\end{pmatrix}\pmod{m}$. 
We have not used the usual diamond notation for these operators, since that notation
is reserved for the Petersson product below.

\subsubsection{Atkin-Lehner-Li operators}
\label{atkinlehner}
 Let $m$ denote any positive integer, and let $m=QQ'$ with $Q, Q'\geq 0$, 
be a factorization such that $(Q, Q')=1$. 
We let $W_{Q}$ and $W_{Q'}$ denote the Atkin-Lehner-Li operators defined 
by Atkin-Li in \cite{ali}, top of page 223, acting on the modular forms for the group $\Gamma_1(m)$ . 
Explicitly, the operator $W_Q$ is given by the action of any matrix $\gamma = \begin{pmatrix} Qx & y\\Nz &Qw
\end{pmatrix}$, where $x, y, z, w\in \mathbb{Z}$, and $y\equiv 1\pmod{Q}, x\equiv 1\pmod{Q'}$, and
$\det(\gamma)=Q$. 
The operator $W_m$ defined above corresponds to the factorization $m= m\cdot 1$. 
Suppose that $F$ has nebentype character $\epsilon$ mod $m$, and write $\epsilon=\epsilon_Q\cdot \epsilon_{Q'}$, 
where $\epsilon_Q$ (resp. $\epsilon_{Q'}$) is a character defined mod $Q$ (resp. defined mod $Q'$). 
Then by \cite{ali}, Proposition 1.4, we have 
$W_m = \epsilon_{Q'}(Q)\cdot W_{Q}\circ W_{Q'}$. In particular, $W_Q$ may not commute with 
$W_{Q'}$.  

The operator $W_Q$ depends on the level $m$, not just on the index $Q$,
 although it is not customary to put that in the notation. We will put some indication of the level if it is required.
 However, for $Q, Q'$ as above, the matrix written above which defines $W_{Q}$
 at level $m$  also satisfies the definition of the matrix giving $W_{Q}$ at level $Q$ corresponding to $Q = Q\cdot 1$.
 Thus, if we view the forms of level $Q$ as a subspace of the forms of level $m$, via the natural inclusion
 of functions on the upper half plane, the action of the operator
 $W_{Q}$ of level $Q$ coincides with the action of $W_{Q}$ at level $m$.
  We will use this fact later. Note also that the eigenvalues of the $W$ operators on a given form $f$ may not be contained in the 
   ring generated by the Fourier coefficients of $f$; thus we enlarge the coefficient ring to contain these eigenvalues if required.

\subsection{Modular forms and their depletions}
\label{depletion definition}
With the definitions in hand, we return to the situation of $p$-stabilized newforms, as in the introduction. 
In this section, we will be dealing exclusively with the construction and integrality of $p$-adic
L-functions for a single modular form, so we simply fix some $g=g_i$.  Write $M$ for the level
of $g$, so that $M=M_0p$, where $p\nmid M_0$.
Let $\epsilon$ denote the character of $g$, and denote the newform associated to $g$ by $g_0$,
so that $g_0$ has level 
$M_0$ or $M_0p$. 

\par If $z$ denotes a variable in the upper
half plane, write the Fourier expansion of $g$ as $g(z)=\sum a(n,g)e^{2\pi in z}$. 
 Let $S_0$ be a finite set of primes $q\neq p$.
 Define the modular form 
 \[f(z) :=\sum_{(n,S_0)=1}a(n,g) e^{2\pi inz},\] where the sum is restricted to indices $n$ that 
 are indivisible by each prime in $S_0$. 
 Then the $L$-function $L(f, s) = \sum a(n,f) n^{-s}$ of $f$ has the formal Euler product expansion
\[L(f, s) = \prod_q(1-\alpha_q' q^{-s})^{-1}(1-\beta_q' q^{-s})^{-1},\]where the product is taken over all prime numbers $q$, and $\alpha_q', \beta_q'$ are certain complex numbers. The basic
properties of $f$ are well-known, and we summarize them as a lemma.

\begin{Lemma}
\label{depletion lemma}
The modular form $f$ has level
$N$, where $\ord_q(N)=\ord_q(M)$ for all $q\notin S_0$. If $q\in S_0$, then
 $\ord_q(N)=\ord_q(M)+1$ if $q$ is ordinary for $g$.  If $q\in S_0$ is unramified
 for $g$, then
 $\ord_q(N)=\ord_q(M)+2$. Finally, if $q\in S_0$ is depleted for $g$, then
$\ord_q(N)=\ord_q(M)$.
 
Furthermore, the following statements hold:
\begin{enumerate}
\item $N=pN_0$, where $(N_0, p)=1$;
\item $4\vert N$ if $2\in S_0$; 
\item $f$ is an eigenvector
of the Hecke operators $T_q$ for $q\nmid N$ and for $U_q$ when $q\vert N$;
\item the eigenvalue $\alpha_p$ of the $U_p$ operator  on $f$ is a $\mathfrak{p}$-adic unit;
\item  the eigenvalue of $U_q$ on $f$ is zero, for $q\mid N, q\neq p$; and
\item $\alpha_q'=\beta_q'=0$ for all $q\mid N, q\neq p$. 
\end{enumerate}
\end{Lemma}

 

\subsection{The naive symmetric square, and the Petersson product formula} 

We keep the notations for the modular form $g$ and its depletion $f$,
as in the  previous 
subsection. Consider a fixed character $\psi$, of conductor $c_\psi$ relatively prime to $p$. 
Let $\chi$ be a Dirichlet character of the form $\psi\eta$, 
where $\eta$ has $p$-power conductor. Then the naive  $\chi$-twisted symmetric square $L$-function of $f$  is given by:
\begin{equation}
\label{naive-product}
D_f(\chi, s) = \prod_{q} \left( (1-\chi(q)\alpha_q'\beta_q' q^{-s})(1-\chi(q)\beta_q'^2q^{-s})(1-\chi(q)\alpha_q'^2q^{-s})\right)^{-1}.
\end{equation}
The product is taken over all primes $q$, with $\chi(q)=0$ for $q\vert c_\chi =c_\psi p^r$, for $r\geq 0$. It converges in a suitable 
half-plane, and admits a meromorphic continuation to all $s$. 
Let  $G({\chi})$ denote the Gauss sum of $\chi$. 
Then the quantity
\begin{equation}
D_f(\chi, s)^\text{alg} = \frac{D_f(\chi, s)}{\pi^{k-1}\langle f, f\rangle_N}\cdot \frac{G(\overline{\chi})}{(2\pi i)^{s-k+1}}
\end{equation} is algebraic when $s=n$ is an integer in the range
$1\leq n\leq k-1$
satisfying $(-1)^n=-\chi(-1)$. This is well-known, see \cite[Theorem 2.2.3]{lz16}, for the present formulation; the result goes back to
Shimura, whose method was elaborated by Schmidt \cite{schmidt86}, \cite{schmidt88},  and Sturm \cite{sturm}.
In this situation we say that $n$ is critical.
We remark that the functional equation for the primitive symmetric square L-function leads to similar algebraicity results
for $D_f(\chi, s)^{\op{alg}}$ for integer values of $s$ in the range $k\leq s\leq 2k-2$; we will not need these results here.
Note that the algebraic quantities above are not necessarily integral.  
We make the following assumption, to keep the notation and book-keeping simple:

\begin{itemize}
\item we have $\chi=\psi \eta$, where $\eta$ is a \emph{nontrivial} even character of $p$-power conductor, and
\item $s=n$ is an odd integer with $1\leq n \leq k-1$ in the algebraicity formula above. 
\end{itemize}

\begin{Remark} The first assumption above would be problematic if one were trying to \emph{construct} a $p$-adic L-function,
since in that case one must consider the trivial character.
 However, we are simply trying to prove congruences for a $p$-adic L-function that is known to exist, so it suffices to demonstrate
 the congruences for almost all characters. 
\end{Remark}

\subsection{Explicit formulae for L-values}
\label{shimura formula}
The explicit formulae we will need  for congruences originate in \cite{shimura-holo}. 
We have elected to follow  the treatment in \cite{schmidt86}, with some improvements, 
since it is relatively easy to compare the formulae given there to those
originally given by Shimura, whose work remains the basic reference. But before giving Shimura's
formula, we need to clear up an important technical point. Shimura considers the convolution
of $f$ with a modular form of half-integral weight, whose level is therefore divisible by $4$. It will be convenient
for the computations to insist that the level $N$ of $f$ is also divisible by $4$, and we can achieve this by requiring
$2\in S_0$. Thus, until further notice, we assume that 
\begin{itemize}
\item The set $S_0$ contains $q=2$.
\end{itemize}
We will remove this condition on $S_0$ at the very end of the argument, when we give the proofs of the theorems stated
in the introduction. Thus, define
\begin{equation}\theta_\chi(z) = \frac{1}{2}\sum_{j\in\mathbb{Z}} \chi(j) \text{exp}(2\pi i j^2 z), 
\end{equation} which is a modular form of weight $1/2$ and level $4c_\chi^2$. 
Further, let $$\Phi^\epsilon(z, \chi, s) = L_{N_\chi}(\chi^2, 2s+2-2k)E(z, s+2-2k, 1-2k, \omega)$$ denote the 
Eisenstein series as defined in \cite[p.210]{schmidt86}.  We reproduce the formula here:
$$E(z, s, \lambda, \omega) = y^{s/2}\sum_\gamma \omega(d_\gamma)j(\gamma, z)^\lambda|j(\gamma, z)|^{-2s},$$
where $z, s, \lambda\in\mathbb{C}$, and $\omega$ is the Dirichlet character given by  
\begin{equation}
\label{omega}
\omega(d) = \epsilon(d) \chi(d) \left(\frac{-1}{d}\right)^k
\end{equation}
for $(d, 4Nc_\chi)=1$. The sum is taken over $\gamma\in\Gamma_\infty\backslash\Gamma_0(N_\chi)$,
where  $N_\chi:=\text{lcm}(N, 4c_\chi^2)$.
The subgroup $\Gamma_\infty$ is the set of matrices of the form $\pm \begin{pmatrix}1& n \\0 & 1\end{pmatrix}, 
n\in\Z$, and $d_\gamma$ denotes the bottom right entry of the matrix $\gamma$.

The Eisenstein
series is a real analytic modular form of weight $k-1/2$, and $\theta_\chi(z)\Phi^\epsilon(z, \chi, s)$ is a (non-holomorphic) modular form of weight $k$,  character $\epsilon$, and level 
 $N_\chi$. 
 Recall that the level $N$ of $f$ is divisible by precisely the first power of $p$, since $f$ is assumed to be $\mathfrak{p}$-stabilized.
We shall write $c_\eta=p^{m_\chi}=p^{m_\eta}$ for the $p$-part of the level of $\chi=\psi\eta$. Since $\eta\neq 1$
we have $m_\chi\neq 0$ and $N_\chi=
\op{lcm}(N, 4c_\psi^2p^{2m_\chi})=p^{2m_\chi}\op{lcm}(N, 4c_\psi^2)$.

Consider an odd integer $n$ in the range
$1\leq n\leq k-1$. Shimura has
 shown that $\theta_{\overline\chi}(z)\Phi^\epsilon(z,\overline\chi, n)$ is a nearly holomorphic modular form of level
$N_\chi$, weight $k$, and trivial character. In fact, one has the following formula (see equation (1.5) in \cite{shimura-holo}):
\begin{equation}\label{sturm relation}
(4\pi)^{-s/2}\Gamma(s/2) D_f(\chi, s) = \frac{1}{\phi(N_\chi)}\langle f, \theta_{\overline\chi}(z)\Phi^\epsilon(z,\overline\chi, s)\rangle_{N_\chi}.
\end{equation}

Here the inner product is taken as an integral over a fundamental domain for $\Gamma_1(N_\chi)$. 
The factor of $\frac{1}{\phi(N_\chi)}$ is inserted to match Shimura's integral
over a fundamental domain for $\Gamma_0(N_\chi)$. 

\subsection{Trace computations and the Petersson product formula at level $N$}

For the purpose of proving congruences, it is vital to work at the fixed level $N$. However, observe that
the form $\theta_{\overline\chi}(z)\Phi^\epsilon(z,\overline\chi, n)$ satisfies a transformation property with respect to the smaller group 
$\Gamma_0(N_\chi)$, which varies with $\chi$.  
Our goal is to bring $\theta_{\overline\chi}(z)\Phi^\epsilon(z,\overline\chi, n)$ down to level $N$ by taking a trace, and eventually verify
that we can retain
control of integrality of the Fourier coefficients. In this section we define and compute the trace, and
prepare for the discussion of holomorphy and integrality in the next section. 

Let $N_\psi=\op{lcm}(N, 4c_\psi^2)$. The level $N$ of $f$ and the level $N_\chi=p^{2m_\chi}N_\psi$ of 
$\theta_{\overline\chi}(z)\Phi^\epsilon(z,\overline\chi, n)$ 
differ only by a power of $p$, and other primes $q\vert c_\psi$.  It may happen that the primes dividing $c_\psi$ also divide $N$.

We start by dealing with the powers of $p$. In this we reproduce the method of  Schmidt \cite{schmidt86}, \cite{schmidt88}, 
with a few improvements. 
Let $T_\chi$ denote the trace
operator that takes modular forms on $\Gamma_1(N_\chi)$ down to $\Gamma_1(N_\psi)$. 
The formula for the trace on $\Gamma_1$ may be deduced from the formulae on pages 68-69 of Ohta \cite{ohta95}, especially (2.3.1) on
the top of page 69, where coset representatives 
for $\Gamma_1(N_\psi p^{r+1})$ to $\Gamma_1(N_\psi p^{r})$ are computed, for any $r\geq 1$. Ohta's notation is different
from ours -- his $N$ is a number prime to $p$, which corresponds to our $N_\psi/p$. Let $F$ denote a modular form 
on $\Gamma_1(N_\chi)$. Assume that the character of $F$ has conductor dividing $N_\psi$. In the following pages,
we will
use the symbol $\circ$ to denote the action of operators on the right
(instead of the usual slash) because the formulae are somewhat messy, and
 $\circ$ renders the various compositions of operators more readable.

With this convention, Ohta's formulae (2.3.2) and (2.3.3) \emph{op. cit.} 
yield $$p^{k/2-2} F\circ \op{Tr}^{\Gamma_1(N_\psi p^r)}_{\Gamma_1(N_\psi p^{r-1})}  = F\circ W_{N_\psi p^r}\circ U_p \circ W_{N_\psi p^{r-1}}^{-1}.$$
Here we have used the fact that the nebentype character is defined modulo $N_\psi$ to conclude that the $p$ matrices denoted by $\sigma_\alpha$
in Ohta act trivially, together with the definition of $U_p$ and the slash operator in Section \ref{assumptions-and-definitions}.  Iterating this relationship yields the following formula, which is valid for any $F$ as above:
\begin{equation}
\label{annoying formula}
p^{(2m_\chi-1)(k/2-2)} F\circ T_\chi\circ W_{N_\psi}=
F\circ W_{N_\chi}\circ U_p^{2m_\chi -1}.
\end{equation}
We remark that $T_\chi$  is defined purely in terms of matrices, and can be applied to the non-holomorphic form 
$\theta_{\overline\chi}(z)\Phi^\epsilon(z,\overline\chi, n)$,
and that the latter has character $\epsilon$, which is defined modulo $N$.

Thus we may apply our formula to $F= \theta_{\overline\chi}(z)\Phi^\epsilon(z,\overline\chi, n)$. As a result, we find that 
$\theta_{\overline\chi}(z)\Phi^\epsilon(z,\overline\chi, n)\circ W_{N_\chi}\circ U_p^{2m_\chi-1}$ is of level $N_\psi$. 
It follows further that for any $m\geq m_\chi$,
that $\theta_{\overline\chi}(z)\Phi^\epsilon(z,\overline\chi, n)\circ W_{N_\chi}\circ U_p^{2m -1}$ is also a modular form of level $N_\psi$, since the $U_p$ operator at level $N_\chi$ is 
given by the same  matrices as $U_p$ at level $N_\psi$, and $U_p$
stabilizes the space of forms of level $N_\psi$. Since \[\theta_{\overline\chi}(z)\Phi^\epsilon(z,\overline\chi, n)\circ W_{N_\chi}\circ U_p^{2m -1}= 
\theta_{\overline\chi}(z)\Phi^\epsilon(z,\overline\chi, n)\circ W_{N_\chi}\circ U_p^{2m_\chi-1} \circ U_p^{2(m-m_\chi)-1},\] 
the equation \eqref{sturm relation} gives
\begin{align*}
\frac{\Gamma(n/2) D_f(\chi,n)}{(4\pi)^{n/2}} = & \frac{1}{\phi(N_\chi)} \langle f,  \theta_{\overline\chi}(z)\Phi^\epsilon(z,\overline\chi, n) \rangle _{N_\chi} \\
=  &  \frac{1}{\phi(N_\chi)} \langle f,  \theta_{\overline\chi}(z)\Phi^\epsilon(z,\overline\chi, n)\circ T_\chi \rangle _{N_\psi}  \\
= &  \frac{1}{\phi(N_\chi)}\langle f\circ W_{N_\psi},  \theta_{\overline\chi}(z)\Phi^\epsilon(z,\overline\chi, n)\circ T_\chi\circ W_{N_\psi}\rangle_{N_\psi}\\
=&  \frac{p^{-(2m_\chi-1)(k/2-1)}}{\phi(N_\psi)}\langle f\circ W_{N_\psi},  \theta_{\overline\chi}(z)\Phi^\epsilon(z,\overline\chi, n) \circ W_{N_\chi}\circ U_p^{2m_\chi -1}\rangle_{N_\psi}.
\end{align*}
Here we have used \eqref{annoying formula}, as well as the factorization $\phi(N_\chi) = \phi(N_\psi)p^{2m_\chi-1}$.

Computing a bit further, and setting $H=\theta_{\overline\chi}(z)\Phi^\epsilon(z,\overline\chi, n)$ for brevity,
we consider some $m\geq m_\chi$, and we find that
\begin{align*}
\langle f\circ W_{N_\psi}, H\circ W_{N_\chi}\circ U_p^{2m-1}\rangle_{N_\psi} 
=  &\langle f\circ W_{N_\psi}, H\circ W_{N_\chi}\circ U_p^{2m_\chi-1} \circ U_p^{2(m-m_\chi)}\rangle_{N_\psi} \\  
= & \langle f \circ W_{N_\psi}\circ (U_p^*)^{2(m-m_\chi)},H\circ W_{N_\chi}\circ U_p^{2m_\chi-1}\rangle_{N_\psi}  \\
= & \langle f\circ U_p^{2(m- m_\chi)}\circ W_{N_\psi}, H\circ W_{N_\chi} \circ U_p^{2m_\chi-1}\rangle _{N_\psi}\\
= & \alpha_p^{2(m-m_\chi)} \langle f\circ W_{N_\psi}, H\circ W_{N_\chi}\circ U_p^{2m_\chi-1}\rangle _{N_\psi}.
\end{align*}

In the second calculation above, we have used the fact that adjoint of $U_p$ at level $N_\psi$ is given by 
$U_p^* = W_{N_\psi}^{-1}\circ U_p\circ W_{N_\psi}$. 
Putting together the strings of equalities above, we conclude that
 \begin{equation}
 \label{petersson-formula-first}
 \begin{split}
&\phi(N_\psi)\frac{\Gamma(n/2)}{(4\pi)^{n/2}}p^{(2m_\chi-1)(k/2-1)} D_f(\chi,n) \\
= & \alpha_p^{2(m_\chi-m)} \langle f\circ W_{N_\psi}, \theta_{\overline\chi}(z)\Phi^\epsilon(z,\overline\chi, n)\circ W_{N_\chi}\circ U_p^{2m-1}\rangle_{N_\psi}
\end{split}
\end{equation}
for any $m\geq m_\chi$. 

\begin{Remark} The factor $\phi(N_\psi)$ in the formula above depends on $\psi$, and may be divisible by additional powers of $p$, coming 
from primes $q\equiv 1\pmod{p}$.
\end{Remark}


We want to replace the inner product in the formula above by an inner product at the fixed level $N$. Thus we have to take
another trace. We warn the reader that the calculation is a bit messy, although more or less elementary. But we need
to do the calculation in detail, as it turns out to be what we need to get rid of the extra powers of $p$ in the 
quantity $\phi(N_\psi)$ that come from primes $q$ with $q\vert c_\psi, q\nmid N$, and $q\equiv 1\pmod{p}$.

Let $N=\prod_{i=1}^s q_{i}^{e_i}$ be
the prime factorization of $N$. Reordering the factors if necessary, suppose that $q_i\nmid 2c_\psi \iff i <  t$. If $i < t$, then clearly $\op{ord}_{q_i}(N_\psi)
=\op{ord}_{q_i}(N)$, because $N_\psi = \op{lcm}(N, 4c_\psi^2)$.
The level $N_\psi=\op{lcm}(N, c_\psi^2)$ differs from $N$ only at the primes $q_i$ with $i\geq t$, where 
$f_i = \op{ord}_{q_i}(N_\psi) \geq e_i = \op{ord}_{q_i}(N)\geq 1$, or at primes which divide $4c_\psi$ which do not divide $N$ at all.

To start with, we want to get rid of the extra powers of the $q_i, i\geq  t$, so we define $N' _\psi= {N_\psi}/(\prod_{i\geq t} q_i^{f_i-e_i})$ and set
 $$U^\psi_1 = U_{q_t}^{f_t-e_t}\circ\dots  U_{q_s}^{f_s-e_s}.$$  
If we set $N_{q_t}= N_{\psi}/q_t^{f_t-e_t}$, then the same 
calculation as in the case $q=p$ treated above shows that, if $F$ is a form of level $N_\psi$ whose nebentype character has conductor
dividing $N$,  we have
$$F\circ \op{Tr}_{q_t}  \circ W_{N_{q_t}}= c_{q_t}F\circ W_{N_\psi}\circ U_{q_t}^{f_t-e_t} .$$ 
Here $c_{q_t}$ is a power of $q_t$, and $\op{Tr}_{q_t}$ is the trace to level $\Gamma_1(N_{q_t})$. 
By construction, we have $\op{ord}_{q_t}(N_{q_t})=\op{ord}_{q_t}(N)$. Repeating this procedure,
we conclude that $ F \circ U^\psi_1$ has level $N_\psi'$. 
Furthermore, it is clear, by induction, that if $G$ is a cuspform of level $N'_\psi$,  and $F$ has character
with conductor dividing $N'_\psi$, that we have
\begin{equation}
\label{whew}
\langle G, F \circ W_{N_\psi}\circ U^{\psi}_1\circ W_{N_\psi'}^{-1}\rangle_{N'_\psi} = a_\psi \langle G, F\rangle_{N_\psi},
\end{equation}
where $a_\psi$ is a $\p$-adic unit.
Let $Q = \prod_{q\vert N} q^{{\op{ord}_q(N_\psi'})}$, and set $Q'=N'_\psi/Q$. 
Then the following lemma holds by construction:
\begin{Lemma}
\label{al lemma}
We have $Q=N$, and a factorization $N_\psi'= QQ'$ where $(Q', N) =1$.
\end{Lemma}

Let $W_{Q}$ and $W_{Q'}$ be the Atkin-Lehner-Li operators at level $N_\psi'$.  
As we have already recorded
in Section \ref{atkinlehner}, we have  $W_{N}\circ W_{Q'}=W_{Q}\circ W_{Q'}= \epsilon_{Q'}(Q)W_{Q'Q}=
W_{N_\psi'}$,
since $\epsilon$ is defined modulo $Q =N$, on forms
 with character $\epsilon$ and level $N_\psi'= QQ'$. In the situation at hand where $Q=N$, the operator
 $W_{Q}$ acting on forms of level $N_\psi'$
agrees with $W_{N}$ defined in the same way on the subspace of forms of level $N$.

Now, let $U^\psi_2$ denote the trace operator from level $N_\psi'$ to level $N$.
We define
$\tilde T_\psi: S_k(N_\psi, K) \rightarrow S_k(N, K)$ by 
$$\tilde T_\psi: h\mapsto  h\circ U_1^\psi \circ W_{Q'}^{-1}\circ U_2^\psi.$$ 
The reason for the factor $W_{Q'}^{-1}$ will become clear very shortly.
 The key lemma is the following. 

\begin{Lemma}
\label{tame-trace}
Suppose $F$ is a holomorphic modular form of level $N_\psi$, character $\epsilon$ modulo $N$, and weight $k$,
with $\p$-integral Fourier coefficients. Assume that $\epsilon$ has trivial restriction to $(\zz/p\zz)^\times$,
namely, that $\Gamma$ is modular with respect to $\Gamma_1(a, p)=\Gamma_1(a)\cap\Gamma_0(p)$, for $a=N_\psi/p$.
Then the following statements hold:
\begin{enumerate}
\item $F\circ \tilde T_\psi$ has $\p$-integral Fourier
coefficients, and 
\item The Fourier coefficients of $ F\circ \tilde{T}_\psi $ are divisible by $\prod_{q} (q-1)$, where the product is taken over all primes $q$ dividing
$c_\psi$ which do not divide $N$.
\end{enumerate}
\end{Lemma}

\begin{proof} We start with the first statement. 
The fact that $U_{\psi}^1$ preserves integrality is clear from the definition, virtue of the standard expression
for the action of the $U_q$ operators on the Fourier coefficients. The remainder of the proof
 boils down to more or less well-known facts
about the special fibre of the modular curve $X_1(a, p)$, so we merely sketch the details.
For $W_{Q'}$, this is treated in the case 
of modular forms with trivial central character 
 by Theorem A.1 in the appendix by Conrad to \cite{pras09}, since $(Q', p) =1$. 
 Conrad only treats the case of $\Gamma_0$ level structure. However, an examination of Conrad's argument 
 show that it works equally well
 for modular forms on $\Gamma_1(a)\cap \Gamma_0( p)$. The key point is that the special fibre 
 of the moduli stack of $X_0(ap)$ in characteristic $p$ has a distinguished irreducible component, 
  and that this component contains both the cusp at infinity as well as
 its image under the operator  $W_{Q'}$ (because $p\nmid Q'$).  Conrad then 
 applies the Katz
 $q$-expansion principle on this component. 

 The existence of the distinguished component is unaffected
 by the auxiliary level structure at primes away from $p$. Indeed, it is obtained as 
 the image of a section $X_1(a) \rightarrow X_1(a, p)$ over the ordinary locus, which enhances an \emph{ordinary}
 elliptic curve, plus any given level structure at $a$, with the level structure at $p$ coming from the unique subgroup
 order $p$ isomorphic to $\mu_p$ (which exists because of ordinariness). 
 We remark also that Conrad's
 geometric definition of the $W_{Q'}$ operator at the bottom of page 60,
 also carries through to our situation, since $(Q', p)=1$ as in Conrad, and
  the $p$-part of our moduli problem is exactly the same. 
 The prime-to-$p$ part of the level causes no problems, since we are working over a ring in which
 $a$ is invertible. Thus, we obtain the our result for $W_{Q'}$ by applying the $q$-expansion principle
 on the distinguished component: if a modular form vanishes identically in a formal neighbourhood of one
cusp, it must vanish on the whole component, and hence the expansion is zero about the other cusp as well.
 
As for $U^2_\psi$, the proof is implicit in Hida's work (see  \cite{hida88}, IV, page 11) and is 
in fact very similar to the above. The trace is given by $F\mapsto \sum_\gamma F\circ\gamma$, where
$\gamma$ runs over a set of coset representatives of $\Gamma_1(N_\psi')\backslash\Gamma_1(N)$. By definition of $N$ we have 
$\Gamma_1(N)\subset \Gamma_1(p)$.  Thus the cusp $s=\gamma(\infty)$ is one of Hida's `unramified at $p$' cusps 
(see \cite{hid86}, Section 5) and it is well-known that if $F$ has
integral $q$-expansion at $\infty$ then it has integral $q$-expansion at $s$. Again, 
the key point is that both $\infty$ and $s$ reduce modulo $\p$ 
to points on the same component of the special fibre, just as in Conrad's argument. 
Thus $t_\psi$ preserves integrality as well.

To prove the second statement in the lemma, let $N_q'$ denote the prime-to-$q$-part of the level $N'_\psi$, 
for $q$ as in the statement of the Lemma. Since the character $\epsilon$ has conductor dividing $N$, and
$q$ does not divide $N$,  the form $F$ is invariant under the group $\Gamma_q=\Gamma_0(q^{e_q})\cap\Gamma_1(N_q')$. 
Since we have $\Gamma_1(N'_q)\supset\Gamma_q\supset \Gamma_1(N_\psi')$, and the latter index is $(q-1)q^{e_q-1}$, the trace
to $\Gamma_1(N_q')$ is divisible by $q-1$. The claim now follows by iterating this process with $N_q'$ instead of $N$, and
any other prime $q$ appearing in the statement of the Lemma. 
\end{proof}

Let $b_\psi =  \prod_{q} (q-1)$, where the product is taken over the same primes $q$ as
in the statement of the lemma, and 
define the operator $T_\psi$ by $T_\psi = \frac{1}{b_\psi}\tilde{T}_\psi$. The operator $T_\psi$ will be our tame trace operator. The lemma
above only applies to holomorphic forms, but the operator $T_\psi$ can be applied even in the nearly holomorphic case, since 
it just involves matrices. We compute the result for later use. 
If $F$ is a cuspform of level $N$, and $G$ is a nearly holomorphic form with character defined modulo $N$, and $m\geq m_\chi$,
then plugging in the definitions, and using Lemma \ref{al lemma} and the remarks that follow, shows that we have
\begin{equation}
 \begin{split}
 \label{trace-petersson-formula}
b_\psi\langle F\circ W_N ,  G\circ U_p^{2m-1}\circ T_\psi \rangle_N= &
  \langle F\circ W_N ,  G\circ U_p^{2m-1}\circ U_1^\psi\circ W_{Q'}^{-1}\circ U_2^\psi \rangle_N\\
=&  \langle F\circ W_N ,  G\circ U_p^{2m-1}\circ U_1^\psi\circ W_{Q'}^{-1}\rangle_{N_\psi'}\\
= & \langle F,  G\circ U_p^{2m-1}\circ U_1^\psi\circ W_{N_\psi'}^{-1}\rangle_{N_\psi'}\\
=&  \langle F,  G\circ U_p^{2m-1}\circ W_{N_\psi}^{-1}\rangle_{N_\psi}.
 \end{split}
 \end{equation}
Taking $G=\theta_{\overline\chi}(z)\Phi^\epsilon(z,\overline\chi, n)\circ W_{N_\chi}$ and $F=f$ in the above, using the definition of $b_\psi$, and comparing with \eqref{petersson-formula-first},
we obtain the desired inner product formula at level $N$:
 \begin{equation}
 \label{petersson-formula}
 \begin{split}
&\phi(N)\frac{\Gamma(n/2)}{(4\pi)^{n/2}}p^{(2m_\chi-1)(k/2-1)} D_f(\chi,n) \\
= &\text{unit} \cdot \alpha_p^{2(m_\chi-m)} \langle f\circ W_N, \theta_{\overline\chi}(z)\Phi^\epsilon(z,\overline\chi, n)\circ W_{N_\chi}\circ U_p^{2m-1}\circ T_\psi\rangle_{N}
\end{split}
\end{equation}

The unit in this formula depends only the number $c_\psi$ and the character $\epsilon$. It could be specified exactly, but we will not need it. 

Finally, we observe that the number $\phi(N)$ may still be divisible by $p$, in the event that $N$ itself is divisible by primes $q$ with $q\equiv 1\pmod{p}$. 
To get rid of these, let $U$ denote the $p$-Sylow subgroup of $(\zz/N\zz)^\times$. Since
$N$ is divisible only by the first power of $p$, we see that $U$ is the direct product of the $p$-Sylow subgroups
of $(\zz/q_i\zz)^\times$, where $q_i$ runs over primes dividing $N$ such that $q\equiv 1\pmod p$. In particular,
there is no contribution from $q=p$. 
If  $N_\chi = \prod q^{e_q}$ is the prime factorization of $N_\chi$,  we have 
$(\zz/N_\chi\zz)^\times\cong\prod (\zz/q^{e_q}\zz)^\times$. Therefore
we may identify $U$ with a certain direct summand
 of the $p$-Sylow subgroup of  $(\zz/N_\chi\zz)^\times$. With this in mind, let $\epsilon'$ denote a character of  $(\zz/N\zz)^\times$ of $p$-power
order. We may identify $\epsilon'$ with a character of $U$ in the obvious way, and hence as a character of $(\zz/N_\chi\zz)^\times$. 
For each fixed character $\epsilon$ of  $(\zz/N\zz)^\times$, and
for an auxiliary character $\chi=\psi\eta$ and an odd critical integer $n$ as above, each of which is also fixed, consider
the function
\begin{equation}
\label{gamma1-kernel}H_\chi(n)  =\frac{1}{\vert U\vert}\sum_{\epsilon'}
 \theta_{\overline\chi}(z)\Phi^{{\epsilon'}\epsilon}(z,\overline\chi, n)
\end{equation}
where we sum over all characters $\epsilon'$ of $U$. Then clearly we have $\phi(N) = |U|\cdot \text{unit}$, so that if $f$ is a cuspidal eigenform
of level $N$ with fixed character $\epsilon$,  then we have
 \begin{equation}
 \label{petersson-formula-gamma1}
 \begin{split}
&\frac{\Gamma(n/2)}{(4\pi)^{n/2}}p^{(2m_\chi-1)(k/2-1)} D_f(\chi,n) \\
= &\text{unit} \cdot \alpha_p^{2(m_\chi-m)} \langle f\circ W_N, H_\chi(N)\circ W_{N_\chi}\circ U_p^{2m-1}\circ T_\psi\rangle_{N}
\end{split}
\end{equation}

\subsection{Holomorphic and ordinary projectors}

In this section, we analyze the form $H_{{\chi}}(n)\circ W_{N_\chi}\circ U_p^{2m-1}$ defined above
and show how to replace it with something holomorphic and integral.
 The classical method of going from a nearly holomorphic form to something holomorphic, and which is adopted
in \cite{CS}, \cite{schmidt86}, \cite{schmidt88}, is to 
pass from $H_\chi$ to its so-called  holomorphic projection. This is a bit complicated, since the formulae giving the holomorphic 
projection of a nearly holomorphic form involve
factorials and binomial coefficients, and one cannot easily control the denominators. This is why the results of 
\cite{schmidt86}, \cite{schmidt88} are only stated up to some unspecified rational constant.  

One of the main contributions of this paper is a solution to this problem, using $p$-adic methods: the form $f$
is \emph{ordinary}, so we can replace the nearly holomorphic form with a certain ordinary projection, without losing
any information. This has the significant advantage that the ordinary
projector is denominator-free. In view of Hida's control theorems for ordinary forms, the ordinary projection is automatically holomorphic.
We shall follow this alternative path, but to complete the journey, we have to compute the Fourier expansion of $H_{{\chi}}(n)\circ W_{N_\chi}$. 
Since $H_{{\chi}}(n)$ is a linear combination of products of an Eisenstein series and a theta series, we
work out the expansions of these two first, starting with the Eisenstein series.

Following \cite{schmidt86}, page 213,and Shimura, \cite{shimura-holo}, Section 3, page 86,  let us fix a character $\epsilon$ modulo $N$, 
and write 
$$\Phi^\epsilon\left(\frac{-1}{N_{\chi} z}, \chi, n\right) \cdot (\sqrt{N_\chi}z)^{1/2-k}= \sum_{j=0}^{(n-1)/2}\sum_{\nu=0}^{\infty} (4\pi y)^{-j} d_{j, \nu} q^\nu.$$

\begin{Lemma}
\label{eisen-fourier} Let $\chi=\psi\chi$, as before, and that $(\psi\epsilon)^2\not\equiv 1\pmod{\p}$. 
If $\chi$ is ramified at $p$, the quantities $\frac{\Gamma((n+1)/2)}{\pi^{(1+n)/2}} p^{m_\chi(3-2k+2n)/2}d_{j,\nu}$ are algebraic. Furthermore, the quantities $\frac{\Gamma((n+1)/2)}{\pi^{(1+n)/2}} p^{m_\chi(3-2k+2n)/2}d_{j,\nu}$ are $\pp$-integral
if $\nu>0$.
\end{Lemma}

\begin{proof} The formulae for the $d_{j,\nu}$ may be deduced from those on pages 212-213 of \cite{schmidt86}, whose $n$ is our $\nu$, and whose
$m$ is our $n$. 
 For $\nu>0$,  one obtains
$$d_{j, \nu} = (-2i)^{(k-1/2)}\cdot\pi^{\frac{n+1}{2}}\cdot \nu^{\frac{n-1}{2}} \cdot B_j \cdot {\frac{n+1}{2}\choose j}\cdot N_\chi^{(2k-2n-3)/4}\cdot L_{N_\chi}(n+1-k,\omega_\nu)\cdot\beta(\nu, n+2 -2k),$$
where $B_j=\frac{\Gamma(n/2+1-k+j)}{\Gamma(\frac{n+1}{2})\cdot\Gamma(n/2+1-k)}\in\Q$. The definition of $\beta$ 
may be 
found on page 212 of \cite{schmidt86}, or in Proposition 1of \cite{shimura-holo}.  Note also that $\beta$ depends on $\epsilon$ and $\chi$. The character $\omega$ is defined in (\ref{omega}), 
and $\omega_\nu =\left(\frac{-1}{\cdot}\right)^{k+1}\left(\frac{\nu N_\chi}{\cdot}\right)\omega =\epsilon\chi \left(\frac{-\nu N_\chi}{\cdot}\right)$.
The power of $p$ dividing $N_\chi$ is $p^{2m_\chi}$, and hence is a perfect square, so
 the character $\omega_\nu$ is unramified at $p$ if $(\nu, p) =1$.

As for the constant term, one has $d_{j,0}=0$ unless $j=(n-1)/2$, in which case  
\footnote{We remark that there is a typo in Schmidt, where $L_{N_\chi}(2n+2-2k,\omega^2)$ is written.}
$$d_{(n-1)/2, 0} = (-2i)^{(k-1/2)}\cdot\pi^{\frac{m+1}{2}}\cdot B_j \cdot N_\chi^{(2k-2n-3)/4}\cdot L_{N_\chi}(2n+1-2k,\omega^2).$$

It is clear from the formula
for $B_j$ that $\Gamma((n+1/2))B_j$ is a rational integer, and one knows from properties of Kubota-Leopoldt $p$-adic L-functions
that their special values $L(\omega_\nu, n+1-k)$ are $p$-integral, except in the case $\omega_\nu $ is one of certain even 
powers of the Teichm\"uller character $\eta_1$, and $n+1-k$ is one of certain 
odd negative integers. These bad cases correspond to special values of the $p$-adic 
Riemann zeta function, which is the only $p$-adic Dirichlet L-function with a pole.
We have assumed that $(\psi\epsilon_i)^2\not\equiv 1\mod{\pp}$, so that $\omega_\nu$
is ramified at some prime other than $p$, and 
is never a power of the Teichm\"uller character, for any $\nu> 0$. 
But this means that $L(\omega_\nu, n+1-k)$ occurs
as a value of the $p$-adic zeta function of a nontrivial character,
and it is therefore $\mathfrak{p}$-integral. 
Then the result for $\nu >0$ 
follows upon clearing the powers of $p$ coming from $N_\chi$.  
 \end{proof}

\begin{Remark} The integrality properties of Kubota-Leopodlt
L-functions are summarized in the paper \cite{rib78}.
As for the constant terms with $\nu=0$ in the Lemma above, we cannot exclude the possibility that the L-values arise in interpolation of the Riemann
zeta function without further hypothesis. We do not pursue this point, since it is not needed. \end{Remark}

One has next to compute the Fourier expansion of the quantity $\theta_\chi(-1/N_\chi z)\cdot (\sqrt{N_\chi} z)^{-1/2}$. 
Here the exponent comes from the fact that $\theta_\chi$ is a form of weight $1/2$. The requisite formula may be found
in \cite{shimura-half}, Proposition 2.1, and we record the result here. 
Recall the notations: $N=N_0p$ is an integer divisible by precisely
the first power of $p$, and $\chi$ is a character of conductor $c_\psi p^{m_\chi}$.  
Furthermore, $N$ is divisible by $4$. We let $N'= N_\chi/4c_\chi^2$. Thus $(N', p) =1$, if $\chi$ is ramified.

\begin{Lemma} 
\label{theta-fourier}
Suppose that $\chi$ is ramified. 
We have $\theta_{\chi}(-1/N_\chi z)\cdot (\sqrt{N_\chi} z)^{-1/2} = \theta_{\ol{\chi}}(N'z)\cdot 
\frac{g(\chi)}{ \sqrt{c_\psi p^{m_\chi}}}\cdot i^{3/2}\cdot (N^\prime)^{1/4}$.
\end{Lemma}

\begin{proof} See \cite{shimura-half}, Proposition 2.2.
 \end{proof}


\begin{Corollary} 
\label{first-integrality-formula}
Suppose that $\eta$ is ramified, and let $\epsilon$ be any character modulo $N$ such that $(\psi\epsilon)^2\not\equiv 1\mod{\p}$. Then 
$$ \tilde{H}_\chi(n)^\epsilon=\frac{\Gamma((n+1)/2)}{\pi^{(1+n)/2}} p^{m_\chi(3-2k+2n)/2} \cdot \frac{\sqrt{c_\psi p^{m_\chi}}}{g(\chi)} \cdot \theta_{\overline\chi}(z)\Phi^\epsilon(z,\overline\chi, n)\circ W_{N_\chi}$$
is a nearly holomorphic form of level $N_{\chi}$ whose Fourier coefficients $c_{j,\nu}$ are algebraic. Furthermore, the coefficients 
$c_{j, \nu}$ are $\mathfrak{p}$-integral  for all $\nu$ such that $p\vert \nu$. 
\end{Corollary}

\begin{proof} The first statement is obvious, from the formulae for the coefficients of the theta function and the Eisenstein series.
As the for the second statement, the only issue is the possible denominators coming from the constant terms 
of the Eisenstein series. However, $\theta_\chi = \sum b_r q^r = \sum \chi(m)q^{m^2}$ has $b_n = 0$ once $p\vert n$, and
the second claim follows by multiplying out the Fourier expansions.
\end{proof}

\begin{Remark}  We have not made any analysis of the case where $\eta$ is unramified, although it may be done just as above. The 
exact formulae are slightly different, and we will not need them. As we have already remarked, it suffices, for the purpose of integrality, to show that almost all the values are integral (since we are assuming the existence of imprimitive $p$-adic
L-functions). 
\end{Remark}
Now, we follow (\ref{gamma1-kernel}) and define (for fixed $\epsilon, \psi, n$) 
\begin{equation}
\tilde{H}_\chi(n)  =\frac{1}{\vert U\vert}\sum_{\epsilon'} \tilde{H}_\chi(n)^{\epsilon'\epsilon}.
\end{equation}

\begin{Corollary} Suppose that $(\psi\epsilon)^2\not\equiv 1\pmod{\p}$. If $\eta$ is ramified,
then $\tilde{H}_\chi(n)$ is a nearly holomorphic form of level $N_{\chi}$ whose Fourier coefficients $c_{j,\nu}$ are algebraic. Furthermore, the coefficients 
$c_{j, \nu}$ are $\mathfrak{p}$-integral  for all $\nu> 0$ such that $p\vert \nu$. 
\end{Corollary} 

\begin{proof} We need the analogue of Lemma \ref{eisen-fourier}, which dealt with a single fixed character. With the notation of that Lemma, 
we have to show that $\frac{1}{|U|}\sum_{\epsilon'} L_{N_\chi}(n+1-k, \omega'_\nu)\beta'(\nu, n+2-2k)$ is integral,
for each $\nu > 0$. Here $\omega'_\nu$ and $\beta'(\nu, n+2-2k)$ are  given by the same prescription as $\omega_\nu$ and $\beta(\nu, n+2-2k)$, but
with $\epsilon'\epsilon$ instead of 
$\epsilon$. The quantity $\beta'(\nu, n+2-2k)$ is an integral linear combination of terms of the form $c_a\epsilon'(a)$, for certain integers
$a$, with integral coefficients $c_a$, which depend on $\psi, \epsilon, \nu$.  Observe that 
$\epsilon'$ has $p$-power order, and hence $\epsilon'\epsilon\psi$ is never quadratic and $\omega_\nu'$ is never a power of 
the Teichm\"uller character. 
Furthermore, if $a\in(\zz/N_\chi\zz)^\times$,
then $\frac{1}{|U|}\sum_{\epsilon'}\epsilon'(a)$ is $\pp$-integral, by orthogonality of characters of $U$. Then the corollary
follows from the fact that Kubota-Leopoldt L-functions are given by an integral pseudo-measure with a pole only at the trivial
character, as described in \cite{rib78}. 
\end{proof}

Now we want to pass from $\tilde{H}_\chi(n)$ to something holomorphic, while preserving integrality. 
Let $M_k(N_\chi, \oo)$ denote the space 
of all modular forms of level $N_\chi$ with coefficients in $\oo$. 

\begin{Proposition}\label{ordinary holomorphic proj} Let $e$ denote Hida's ordinary projection operator, acting on $M_k(N_\chi, \oo)\otimes\qbar$. Then 
$\tilde{H}_\chi(n)^{\text{hol}}\circ e\in M_k(N_\psi)\otimes\qbar$ has $\pp$-integral Fourier coefficients and level $N_\psi$. Here $\tilde{H}_\chi(n)^{\text{hol}}$ denotes the holomorphic
projection of $\tilde{H}_\chi(n)$. 
\end{Proposition}

\begin{proof} The easiest way to prove the Proposition is to use the geometric theory of nearly ordinary 
and nearly holomorphic modular forms, as developed by Urban \cite{urban12}. A resumé of Urban's work
 adapted to this setting may be found 
in \cite{ros16}. However, we give a proof along classical lines, for the convenience of the reader.  
Write $\tilde{H}_\chi(n) = \tilde{H}_\chi(n)^{\text{hol}} + \sum h_i$, where each $h_i$ is in the image of a Maass-Shimura
differential operator. It follows from the explicit formulae for the differential operators and the formulae for the Fourier
coefficients of  $\tilde{H}_\chi(n)$ that $\tilde{H}_\chi(n)^{\text{hol}}$ and each $h_i$ has algebraic Fourier coefficients. 
Let $m$ denote a positive integer, divisible by $p-1$, and consider 
$\tilde{H}_\chi(n)\circ U_p^{m}$. Then
one has $$\tilde{H}_\chi(n) \circ U_p^{m} = \tilde{H}_\chi(n)^\text{hol}\circ U_p^{m} + \sum  h_i\circ U_p^{m}.$$ It is well known that $U_p$ multiplies
each $h_i$ by a power of $p$ (For example, see \cite{ros16}), formula above Lemma 2.3) Thus, the quantity on the right converges to $\tilde{H}_\chi(n)^{\text{hol}}\circ e$, as $m$
increases, the convergence being assured by the existence of the ordinary projector.
On the other hand, the quantity on the left is $\p$-integral, by Corollary \ref{first-integrality-formula}. This proves the proposition.  The fact that the level comes down to $N_\psi$ arises from the fact
that $U_p^m$ already brings the level down to level $N_\psi$ for any $m$ sufficiently large, as noted in the course of proving (\ref{petersson-formula}). 
\end{proof}

Finally, we bring the $\p$-integral and holomorphic form $\tilde{H}_\chi(n)^{\text{hol}}\circ e$ all the way down to level $N$. Thus we 
define
\begin{equation}
\label{hchi}
\mathcal{H}_\chi(n) = \tilde{H}_\chi(n)^{\text{hol}}\circ e\circ T_\psi\in M_k(N, \oo)\otimes\qbar
\end{equation}
and
\begin{equation}
\label{hchim}
\mathcal{H}_\chi^m(n) = \tilde{H}_\chi(n)^{\text{hol}}\circ U_p^{2m-1}\circ T_\psi \in M_k(N, \oo)\otimes\qbar.
\end{equation}

Then Lemma \ref{tame-trace}, and the proof of Proposition \ref{ordinary holomorphic proj},  give

\begin{Proposition} The modular form $\mathcal{H}_\chi(n) $ has $\p$-integral Fourier coefficients.
The modular forms $\mathcal{H}_\chi^m(n)$ have $\p$-integral Fourier coefficients for all $m$ sufficiently large.
\end{Proposition}

In the next section, we shall see how to compute the inner product of $f$ with  $\mathcal{H}_\chi(n)$ and derive integrality properties for the
special values of $D_f(\chi, s)$. 

\subsection{Algebraic and analytic inner products}

Our goal is to use (\ref{petersson-formula}) to describe the imprimitive $p$-adic L-function, and show that it behaves 
well with respect to congruences. We accomplish this by combining
the inner product formula with a certain algebraic incarnation of the Petersson inner product formula due to Hida. 

Let $S_k(N, \mathbb{Z})$ denote the space of modular
forms of weight $k$ with rational integral Fourier coefficients. If $R$ is any ring, set $S_k(N, R)= S_k(\Gamma_1(N), \mathbb{Z})\otimes R$.
 Let $\T$ denote the ring generated by
the Hecke operators $T_q, U_q, U_p, S_n$ acting on $S_k(N, \cO)$, and set $\T(R)=\T\otimes R$. 
 Recall that our convention is that Hecke operators act on the 
\emph{right}.

Then the eigenform $f$ determines a ring homomorphism $\mathbf{T}\rightarrow \mathcal{O}$, sending a Hecke operator
$T\in \mathbf{T}$ to the $T$-eigenvalue of $f$. 
Define $\mathcal{P}_f$ to be the kernel of this homorphism. There is a unique maximal ideal $\mathfrak{m}$ 
of $\mathbf{T}$ that contains $\mathcal{P}_f$ and the maximal ideal of $\mathcal{O}$. 
\par There 
is a canonical duality of $\mathbf{T}_\mathfrak{m}$-modules 
between $S_k(N, \cO)_\mathfrak{m}$ and $\mathbf{T}_\mathfrak{m}$, which we view as a pairing
\[S_k(N, \cO)_\mathfrak{m}\times \mathbf{T}_\mathfrak{m}\rightarrow \cO.\]
This pairing is explicitly given by $ (s, t)\mapsto a(1, s|t)\in \cO$, where
$t\in \mathbf{T}_\mathfrak{m}$ and $s \in S_k(N,\cO)$ is a $\cO$-linear combination of elements in $S_k(N,\Z)$, given by the Fourier expansion $s = \sum a(n, s) q^n$, with $a(n,s)\in \cO$.

\subsection{The definition of periods and the integrality of special values.} 
To proceed further, we assume that the Galois representation associated to $f$ at $\mathfrak{m}$ is irreducible,
ordinary and $p$-distinguished. It is well-known
that under these conditions that  $\mathbf{T}_\mathfrak{m}$ is Gorenstein, and isomorphic as a 
left $\mathbf{T}_\mathfrak{m}$-module to $\text{Hom}(\mathbf{T}_\mathfrak{m}, \cO)$.
A proof may be found in \cite[Theorem 2.1 and Corollary 2 (p. 482)]{wil95}. We fix, once and for all, an isomorphism
$$\T_{\mathfrak{m}}\xrightarrow{\sim} \op{Hom}(\T_{\mathfrak{m}}, \cO)$$
of $\T_{\mathfrak{m}}$-modules. This choice is determined up to a unit factor in $\T_{\mathfrak{m}}$. Composing with the canonical
isomorphism above leads to an isomorphism $\T_{\mathfrak{m}}\cong S_k(N, \cO)_\mathfrak{m}$, as modules over $\T_{\mathfrak{m}}$,
determined up to multiplication by a unit.

\par Thus, the space $S_k(N, \cO)_\mathfrak{m}$ is equipped with both a left and right action of Hecke operators; one coming from the classically defined slash 
action of Hecke operators on modular forms, and the other the abstract left action obtained from the 
choice of Gorenstein isomorphism in the previous paragraph. In fact, the left action of $\T_{\mathfrak{m}}$ on $\op{Hom}(\T_{\mathfrak{m}}, \cO)$ 
coincides with the usual right action of $\T_{\mathfrak{m}}$ on $S_k(N, \cO)_{\mathfrak{m}}$. Indeed, the isomorphism 
$\T_{\mathfrak{m}}\xrightarrow{\sim} \op{Hom}(\T_{\mathfrak{m}}, \cO)$ is an isomorphism of $\T_{\mathfrak{m}}$-modules, 
and both are free of rank 1,
so the two actions are seen to coincide. One deduces that there is a duality pairing
\begin{equation}
\label{integral-pairing}
(\;\cdot, \cdot)_N: S_k(N, \cO)_\mathfrak{m}\times S_k(N,\cO)_\mathfrak{m} \rightarrow \cO,
\end{equation} 
which satisfies the equivariance condition $(f_1\vert t, f_2) = (f_1, f_2\vert t)$.
This is Hida's algebraic inner product (see \cite{hida_1993}, Chapters 7 and 8). 
Unlike the usual Petersson product, it is linear in both variables,
 and the Hecke operators are self-adjoint. Note that this pairing is dependent on the choice of isomorphism 
 $\T_{\mathfrak{m}}\xrightarrow{\sim} \op{Hom}(\T_{\mathfrak{m}}, \cO)\cong S_k(N, \cO)_\mathfrak{m}$.

With this machinery in hand, we return to the case of interest, as set out in Section \ref{depletion definition}. 
Thus we consider a fixed $p$-ordinary, $p$-distinguished,
and $p$-stabilized newform
$g$ of level $M$, and a finite set $S_0$ of primes $q\neq p$, including $q=2$. 
We let $f$ denote the depletion
of $g$ at the primes in $S_0$, and write $N$ for the level of $g$, as determined in Lemma \ref{depletion lemma}.
We have a maximal ideal $\frm$ in the Hecke alegbra $\mathbf{T}_\mathfrak{m}(N, \cO)$ corresponding
to $f$ and the fixed prime $\p\mid p$, to which we may apply the considerations above. 

Consider the function 
\begin{equation}
\label{alg-pairing}
\varphi_f: v \mapsto (f, v)_N,
\end{equation}
 for $v\in S_k(N, \mathcal{O})_\mathfrak{m}$.
  Let $f^{\perp}\subset S_k(N, \mathcal{O})_\mathfrak{m}$ denote the kernel of $\varphi_f$, and let $$\eta_f = (f, f)_N = \varphi_f(f).$$ 
 The quantity $\eta_f = (f, f)_N$ is the famous congruence number of Wiles \cite{wil95}. Indeed, it follows from the calculations above that $\eta_f = \varphi_f(\op{Ann}(\op{ker}(\varphi_f))$, 
  which is  one of the equivalent definitions given by Wiles. 
  
  We would like to say 
 that $\eta_f$ is nonzero.       
  \begin{Lemma}\label{Tm is a field} Let $f, g, S_0, N$ be as in Section  \ref{depletion definition}, 
  so that $f$ is the depletion of $g$ 
at some finite set $S_0$ of primes $q\neq p$ (not necessarily containing $q=2$). 
 Let 
  $\mathcal{P}=\mathcal{P}_f$ denote the kernel
  of the homomorphism $\mathbf{T}_\mathfrak{m}(\oo)\rightarrow \mathcal{O}$ associated to $f$. Then, there is an isomorphism
  $\mathbf{T}_\mathfrak{m}(\oo)_{\mathcal{P}}\simeq K$.
  \end{Lemma}
  
  \begin{proof} The proof is elementary, but since the result is important, we write out the details.
   Setting $S:=\mathbf{T}_\mathfrak{m}(\oo)\otimes\Q_p$, we note that $S$ is a finite-dimensional algebra over $K$. Hence, $S$ is
  the product of local rings $R_i$, each of which is finite dimensional over $K$. The ring $R_i$ corresponds to a height one prime ideal 
  of $\mathbf{T}_\mathfrak{m}(\oo)$. The localization of $\mathbf{T}_\mathfrak{m}(\oo)$ at $\mathcal{P}$ is equal to one of the rings $R=R_i$. Thus, we are to show that $R$ is a field. 
 The subalgebra $\mathbf{T}'$ of  $\mathbf{T}_\mathfrak{m}(\oo)\otimes\Q$ generated by the Hecke
  operators prime to the level is semisimple, and hence the product of fields $K_i$, each corresponding to a newform of level $N$. Thus, $K\subset R$ is the image
  of $\mathbf{T}'$ in $R$. Note that the summands $K_i$ are not necessarily in bijection with the summands $R_j$, since $K_i$ could potentially be contained in multiple $R_j$s. We have a homomorphism $\mathbf{T}'\rightarrow K\hookrightarrow R$ with kernel
  $\mathcal{P}'$. By construction, the ideal $\mathcal{P}$ lies above $\mathcal{P}'$.
    
\par Consider the subspace of forms in $S_k(N, \oo)_\mathfrak{m}\otimes\Q_p$ annihilated by the ideal $\mathcal{P}'$. 
  By duality, it suffices to show that this subspace is $1$-dimensional over $K$. To achieve this, recall that 
  the newform associated to $f$ is the form $g_0$ at level $M_0$, which differs from $N$
  only at the prime $p$, and at primes $q\in S = S_0\cup \{p\}$ at which $q$ is either unramified or ordinary.
  Note that if $q\in S_0, q\neq p$, then $U_qf=0$.

  \par The argument follows by adding one prime at a time to the level. Let $q\neq p\in S$, and let
  $N_q = M_0q^{e_q}$, where $e_q=0, 1,2$ depending on whether or not $q$ is depleted, ordinary, or unramified, respectively.
   If $q=p$, let $N_p=M=M_0p$.
  Then consider the space $S_q$ given as follows.
  \begin{itemize}
  \item  If $q=p$, and $g_0$ has level divisible by $p$, then $S_q$ is generated by $g_0=g$. 
  \item If $q=p$, and $g_0$ has level prime to $p$,
  then $S_p$ is spanned by $g_0(z)$ and $g_0(pz)$. 
  \item If $q\neq p$ is ordinary, then $S_q$ is spanned by the  $g_0(z), g_0(qz)$.
  \item If $q\neq p$ and $q$ is unramified, $S_q$ is spanned by \[\{g_0(z), g_0(qz), g_0(q^2z)\}.\]
  \end{itemize}
   Each of these spaces
  is stable under the Hecke operator $U_q$, and is annihilated by $\mathcal{P}'$. The eigenvalues of $U_q$ 
  are given as follows (see \cite{ali}, or \cite{miy89}).
  
  \begin{itemize}
      \item In the first case, $U_p$ has the eigenvalue 
  $\alpha_p$, which is a $\mathfrak{p}$-adic unit.
  \item In the second, the eigenvalues are $\alpha_p$, $\beta_p$, and $\beta_p$ is a non-unit.
  \item In the third, the eigenvalues are $\alpha_q=\sqrt{\epsilon(q)q^{k-2}}$ and $0$.
  \item In the fourth, we have $\alpha_q, \beta_q, 0$, and 
  $\alpha_q\beta_q=q^{k-1}$, so both these numbers are nonzero (and in fact units).
  \end{itemize} 
 We claim that, in each case, the localization of $S_q$ at $\mathfrak{m}$ has dimension $1$. In the case when $q=p$, 
the localization of $S_q$ at $\mathfrak{m}$ is one-dimensional. This is because $U_p\notin\mathfrak{m}$ and one of the two eigenvalues $\alpha_p$ and $\beta_p$ 
is a $p$-adic unit and the other is not. On the other hand, in the case when $q\neq p$,
  $U_q\in \mathfrak{m}$. Since, $U_q$ has a unique non-unit eigenvalue ($q\neq p)$ it follows that the localization of $S_q$ at $\mathfrak{m}$ has dimension $1$.
  \par An iteration of this argument over the primes $q$, using the fact that the level raising operators commute with Hecke operators away from the level, and replacing
  the form $g_0$ with the $1$-dimensional space produced in the previous step, implies that our space is $1$-dimensional. In greater detail, express $S_0$ 
  as $\{q_1,\dots, q_r\}$ and for $m\leq r$, set $S_m:=\{q_1, \dots, q_m\}$ and $N_m$ be the largest divisor of $N$ which is divisible by $M$ and the primes in $S_m$. Assume that the subspace of $S_k(N_m, \cO)_{\mathfrak{m}}\otimes \Q_p$ which is annihilated by $\mathcal{P}'$ is one-dimensional over $K$, and let $g_m(z)$ be a generator of this one-dimensional space. Then, apply the same argument as above to $g_m(z)$ in place of $g_0(z)$, to prove that the subspace of $S_k(N_{m+1}, \cO)_{\mathfrak{m}}\otimes \Q_p$ which is annihilated by $\mathcal{P}'$ is one-dimensional over $K$. This inductive argument completes the proof.
\end{proof}

\begin{Corollary} Let $f, g, S_0, N$ be as in Section  \ref{depletion definition}, so that $f$ is the depletion of $g$ 
at some finite set $S_0$ of primes $q\neq p$, including $q=2$. 
Then the quantity $\eta_f = (f, f)_N$ is is non-zero.
\end{Corollary}

\begin{proof} It suffices to prove the corollary upon extending the pairing to $S_k(N, \oo)_\mathfrak{m}\otimes\Q\cong\oplus R_i$. Let $(\cdot, \cdot)_{i,j}$ be the restriction of the pairing $(\cdot , \cdot)_N$ to $R_i\times R_j$. It follows from Hecke-equivariance that this pairing $(\cdot, \cdot)_{i,j}$ is zero if $i\neq j$. Since the pairing $(\cdot, \cdot)_N$ is non-degenerate, it follows that 
\[(\cdot, \cdot)_{i,i}: R_i\times R_i\rightarrow K\] is non-zero. On the other hand, since Lemma \ref{Tm is a field} gives $R_i\simeq K$
 where $i$ is the local component of interest, it follows from $K$-linearity that $(x,x)\neq 0$ for all $x\in R_i$ such that $x\neq 0$. Note that since $f$ is an eigenform, $f$ is nonzero, and hence, we have that $(f,f)_N\neq 0$.
  \end{proof}

 To continue, let $a$ denote a fixed generator of the rank-$1$ $\mathbf{T}_\frak{m}$-module $S_k(N, \mathcal{O})_\mathfrak{m}$,
and consider the number $\varphi_f(a)= (f, a) \in\mathcal{O}$. Then 
any element of $S_k(N, \mathcal{O})_\mathfrak{m}$ is of the form $t\cdot a$ for $t\in \mathbf{T}_\mathfrak{m}$, so the Hecke equivariance of the pairing shows that \[(f, t\cdot a)_N=(f\vert t, a)_N = a(1,  f\vert t)(f, a)_N.\] In particular, it follows that $f^\perp$ is the submodule 
$\mathcal{P}S_k(N, \mathcal{O})_\mathfrak{m}$, and \[S_k(N, \mathcal{O})_\mathfrak{m}/f^\perp\cong (\mathbf{T}_\mathfrak{m}\otimes 
\mathcal{O})/\mathcal{P}(\mathbf{T}_\mathfrak{m}\otimes 
\mathcal{O})\cong \mathcal{O},\]
where $\mathcal{P}$ is the kernel of the canonical homomorphism 
$\mathbf{T}_\mathfrak{m}(\oo)\rightarrow \mathcal{O}$ associated to $f$. Since $(f, f)_N$ is nonzero, we find that 
$f\notin \mathcal{P}S_k(N, \mathcal{O})_\mathfrak{m}$,  and that the function $\varphi_f$ is determined by the nonzero number 
$\eta_f = (f, f)_N\in \mathcal{O}$.

Next we need to compare the algebraic pairing defined above to the usual Petersson inner product. Thus given a modular form $v(z)=\sum a_nq^n\in S_k(N, \mathbf{C})$,
define $v^c(z)=\sum \overline{a}_n q^n$, where the bar denotes complex conjugation. We define a modified Petersson product on $S_k(N, \mathbf{C})$ by setting
\begin{equation}
\label{modified-petersson}
\{v, w\}_N = \langle v\vert W_N, w^c\rangle_N
\end{equation}
where the pairing on the right is the Petersson product (with our chosen normalization). 
One sees from the definition that $\{\cdot, \cdot\}_N$ is $\mathbf{C}$-linear in both 
variables, and that it satisfies $\{v\vert t, w\}_N= \{v, w\vert t\}_N$, for any Hecke operator $t$, just like the algebraic pairing defined above. 

Recall that $\mathcal{O}$ is the completion of the ring of integers of a number field at a prime $\p$ 
corresponding to an embedding in to
$\mathbf{C}_p$. Since we have identified $\mathbf{C}$ with $\mathbf{C}_p$ at the outset, we 
find that the space $S_k(N, \mathcal{O})_\frm$ is equipped with
two $\mathbf{C}_p$-valued pairings $(\cdot, \cdot)_N$ and $\{\cdot, \cdot\}_N$. Each pairing is bilinear, and renders the Hecke operators self-adjoint. 
Just as in the algebraic case, we have a function $\varphi_f^\infty: S_k(N, \mathcal{O})\rightarrow \mathbf{C}_p$ defined by $v \mapsto \{f, v\}_N$,
and the adjointness implies that the kernel of $\varphi_f^\infty$ is the submodule $\mathcal{P}S_k(N, \mathcal{O})_\mathfrak{m}$. Thus we have two different
$\mathbf{C}_p$-valued functions on the rank 1 $\mathcal{O}$-module $S_k(N, \mathcal{O})_\mathfrak{m}/\mathcal{P}S_k(N, \mathcal{O})_\mathfrak{m}$,
and to compare them, it suffices to evaluate on any given element, say on $f$ itself. One is therefore led to consider
$\{f, f\}_N = \langle f\vert W_N, f^c\rangle_N$. It is not clear from the definition that this number
is nonzero; that it is so follows from the same argument that was used in the algebraic case above.

\begin{Definition}
\label{invariant-period}
Define a period associated to $f$ and the level $N$ via $\Omega_{N} = \frac{\{f, f\}_N}{(f, f)_N}$. 
\end{Definition}

\begin{Remark} As defined, the quotient is a ratio of a complex number with a $p$-adic number. One could remedy
this by making the congruence number in the denominator algebraic, by following Hida's original approach from the 1980s 
where the congruence
number is defined in terms of a certain cup product pairing on cohomology, but that would require a substantial digression. 
Our definition presupposes an identification of $\cc_p$ with $\cc$, and that the period depends
(up to unit) upon the choice of isomorphism $\T_{\mathfrak{m}}\xrightarrow{\sim} \op{Hom}(\T_{\mathfrak{m}}, \cO)$.
However, once this choice is made, the period is defined for all Hecke eigenforms of level $N$
whose localization at $\frm$ is non-zero. 
\end{Remark}

A more serious issue is that this definition depends on the level $N$. We would like to claim
that in fact $\Omega_{N}$ is independent of $N$, and depends only on the 
$p$-stabilized newform $g$, up to a $\pp$-adic unit. More precisely, we would 
like to assert that $$\Omega_N = \text{unit}\cdot  \Omega_{M}$$ where $M=M_0p$.
Here $\Omega_M$ is defined by the same prescription as before:
$$\Omega_M =\frac{\{g, g\}_M}{(g, g)_M}$$
where the pairings at level $M$ are derived from the Gorenstein condition and the modified
Petersson product at level $M$. The construction is easier
in this case, since multiplicity one at level $M$ is automatic.

Unfortunately, we cannot quite prove this claim for general weight $k$. 
The case of weight 2 is known, at least under some hypotheses -- this is due
to Diamond, see \cite[Theorem 4.2]{dpnas},
and relies on Ihara's lemma. While there are various
versions of Ihara's lemma known for weight $k> 2$, the specific version
needed here does not seem to be available.

Thus, we will state the precise variant of Ihara's lemma that we need, and make some remarks
about what is known and what is required. 
We will then prove the independence of the period from the 
auxiliary level under the assumption that a suitable Ihara-type lemma holds.

To set the framework, fix a prime 
$p\geq 5$, and consider integers
$A$, $B$ (the levels), together with an auxiliary odd prime $q\neq p$. 
We assume that $A\vert B$, and that one of 
the two following conditions holds:
\begin{enumerate}
    \item $B =q^2A$, and $(A,q)=1$ (the unramified case), or
    \item $B=qA$, and $q\vert A$ (the ordinary case). 
\end{enumerate}
Let $\mathbf{\T}_A$ and $\T_B$ denote the Hecke rings generated by all the Hecke
operators, including $U_q$ as well as $T_q$, at levels $A$ and $B$ respectively. We work
with the groups $\Gamma_1(A)$ and $\Gamma_1(B)$. 
Let $S(A), S(B)$ denote the lattice of cuspforms of levels $A, B$ respectively
whose Fourier coefficients are in $\oo$. Let $S(A, \cc), S(B, \cc)$ denote
the corresponding complex vector spaces. We have Hecke-equivariant
and $\cc$-bilinear and perfect analytic pairings
$\{\cdot, \cdot\}_A:S(A, \cc)\times S(A, \cc)\rightarrow\cc$ and 
and $\{\cdot, \cdot\}_B: S(B, \cc)\times S(B, \cc)\rightarrow\cc$,
defined as above. 
Then let $L_A, L_B$ denote the lattices in $S(A, \cc), S(B, \cc)$ that are
$\oo$-dual to $S_A, S_B$ respectively.  Namely, we have $x\in L_A$ if and only
if $\{x, s\}_A\in\oo$, for all $s\in S_A$, and similarly for $S_B, L_B$. Then
we have $L_A\cong \T_A$ as a $\T_A$-module, and similarly $L_B\cong \T_B$ over
$\T_B$. 

Next, we define a map $\tau:S(A, \cc)\rightarrow S(B,\cc)$, as follows.
Let $h = h(z)\in S(A,\cc)$, where $z$ denotes a variable in the upper half plane.
We define $\tau$ via 
\[\tau\left(h(z)\right) = \begin{cases} h(z) - (U_q h)(qz) &\text{ if }B=Aq,\\
h(z) - (T_q h)(qz) + S_qq^{k-1}h(q^2z) &\text{ if }B=Aq^2.
\end{cases}\]
It is clear that $\tau\left(S(A)\right)\subset S(B)$, and that the image is stable
under $\T_B$. To check the stability under $U_q$, one can calculate explicitly that
$U_q=0$ on the image. Thus $\tau$ is a map that removes the Euler factor at $q$.

Now let $h_A\in S(A)$ be a modular form that is an eigenvector for every 
element $t\in\T_A$. Let $\mathcal{P}_A$ denote the kernel of the homomorphism 
$\phi_A:\T_A\rightarrow \oo$ associated to $h_A$. 
Let $\frm_A$ be the maximal ideal of $\T_A$ corresponding
to the inverse image of the maximal ideal of $\oo$, under $\phi_A$. 
Then, $h_B=\tau(h_A)$ is an eigenvector for $\T_B$ (in fact, with $U_qh_B=0$). 
We may repeat the constructions above for $h_B$, and obtain a height
one prime $\mathcal{P}_B$ and a maximal ideal $\frm_B$ inside $\T_B$.
The ideals $\mathcal{P}_A, \mathcal{P}_B, \frm_A, \frm_B$ are required to satisfy 
additional properties, which we record below:
\begin{itemize}
    \item The localizations of $\T_A, \T_B$ at $\frm_A, \frm_B$ respectively are
    both Gorenstein, and 
    \item The localizations of $\T_A, \T_B$ at $\mathcal{P}_A, \mathcal{P}_B$ are fields isomorphic
    to the fraction field $K$ of $\oo$.
\end{itemize}
We will be applying these considerations to the case that $A$ divides $B$ and $B$ divides $N$, where $N$ is a level derived via
depletion at $S_0$ of a $p$-stabilized newform $g$, as in Section \ref{depletion definition}. Thus the
 maximal ideals we obtain at levels $\frm_A$ and $\frm_B$ are such that
the residual representations at $\frm_A$ and $\frm_B$ are absolutely irreducible and $p$-distinguished (by assumption on $g$). 
It follows from Lemma \ref{Tm is a field} that the second condition will also be satisfied in the case of interest.

With these assumptions in place, we can now state the Ihara-type results that we need.

\begin{hypothesis}
\label{hypothesis ihara} With the conditions and notations above, and any choice 
of $A, B, q$ as above, we have
\begin{itemize}
    \item (Ihara-1) We have $\tau(L_{A, \frm_A})\subset L_{B, \frm_B}$, and
    \item (Ihara-2) $L_{B, \frm_B}/\tau(L_{A, \frm_A})$ is $\oo$-torsion-free.
\end{itemize}
\end{hypothesis}

\begin{Remark}
\label{ihara-comments} 
The lattices $S(A,\cc), S(B, \cc)$ are, by definition, $\oo$-dual to the lattices
of cuspforms of level $A, B$ respectively. It is well-known, in weight 2, that the dual lattices to the space
of integral cuspforms occur in the cohomology of the modular curves $X_1(A), X_1(B)$. 
The map $\tau: S(A,\cc) \rightarrow S(B, \cc)$ is completely
explicit in terms of the usual degeneracy maps of modular curves. Translating the statements
we have written to the language of cohomology gives the familiar Ihara lemma, which was
proven by Wiles in \cite{wil95}, Chapter 2. Wiles's results
were used by Diamond  to prove the Ihara hypothesis as stated here, for the case
$B=q^2A$ (see Theorem of 4.2 \cite{dpnas}). 

The case of weight $k>2$ is a bit more complicated. The cohomology of a modular
curve with coefficients in the Eichler-Shimura module of coefficients may have torsion,
and it is not true in clear how to identify the dual lattice with any simply described
submodule of the cohomology. Furthermore, the duality pairing on cohomology is only defined
with rational coefficients, and does not give any integral duality. However, one can solve these
problems in the ordinary case by using Hida's control theorems to reduce to the case of weight
$2$. For the details of the computation, we refer to the forthcoming thesis of Maletto.
\end{Remark}

%

Now we return to the situation of Definition \ref{invariant-period}. Consider
a $p$-stabilized newform $g$ of level $M$, and the oldform $f$ of level 
$N$ associated to a choice of $S_0$ as before. We want to show that the periods
at level $N$ and $M$ are equal up to a unit. This turns out to be a simple
inductive argument, once Ihara's lemma is known. 

\begin{Lemma} Suppose that the Hypotheses Ihara-1 and Ihara-2 hold, for any
$A, B, q$, with $A\mid B\mid N$.
Then $\Omega_N = u\Omega_M$ for some $p$-adic unit $u$.
\end{Lemma}

\begin{proof} We start at level $M$, and work our way upwards, adding one
prime at time. To spell out the induction, 
we start with a modular form $h_A$ at level $A$, and we move up
to level $B=Aq$ or $Aq^2$, and replace $h_A$ with the $q$-depleted form $h_B$.
In this situation, we are required to show that the periods of $h_A$ and $h_B$
are equal up to a unit. 

It is clear from the definition of the periods that $\Omega_A$ at level $A$ is 
characterized up to a unit by the properties:
\begin{itemize}
    \item $\delta_A:=\Omega_A^{-1}\cdot h_A$ is contained in $L_{A, \frm_A}$,
    \item $L_{A, \frm_A}/\oo\delta_A$
is torsion-free.
\end{itemize} Similarly, the period $\Omega_B$ at level $B$ is characterized
by:
\begin{itemize}
    \item $\delta_B:=\Omega_B^{-1}\cdot h_B$ is contained in $L_{B, \frm_B}$,
    \item $L_{B, \frm_B}/\oo\delta_B$
is torsion-free.
\end{itemize} Since $\tau(h_A)=h_B$  by definition, Ihara-1 
shows that $\delta_B':=\tau(\delta_A)=\tau(\Omega_A^{-1}h_A)$ is contained in $L_{B, \frm_B}$. 
Let $u$ be such that 
$\delta_B=u \delta_B'$. We show that $u$ is a $p$-adic unit. We have that $\delta_B'\in L_{B, \frm_B}$ 
and $L_{B, \frm_B}/\mathcal{O} \delta_B$ is torsion-free. Hence, $u^{-1}$ is contained in $\mathcal{O}$. According to
Ihara-2, $L_{B, \frm_B}/\tau(L_{A, \frm_A})$ is torsion-free. We have that 
$u^{-1}\delta_B=\delta_B'=\tau(\delta_A)\in \tau(L_{A, \frm_A})$. Hence, 
$\delta_B$ is contained in $\tau(L_{A, \frm_A})$, and we write $\delta_B=\tau(\eta_A)$. 
Since the map $\tau$ is injective, it follows that $u^{-1} \eta_A=\delta_A$. 
Since $L_{A, \frm_A}/\mathcal{O}\delta_A$ is torsion-free, it follows that $u\in \mathcal{O}$. 
We have 
shown that $u\in \cO$ and $u^{-1}\in \cO$, so we have deduced that $u\in \cO^\times$.
Therefore $\Omega_A=u\Omega_B$ for some unit $u$.
\end{proof}

\begin{Remark}
To get a nice formula at the end, and to verify that our final
formulae agree with those in \cite{lz16}, we need to further
calculate further. 
We have shown above that the ratio of the algebraic
and analytic pairings at level $M$ and level $N$ are the same. It remains to express everything in terms of the 
newform $g_0$ associated to $f$ and $g$. In other words, we have to bring everything down to level $M_0$. There are
two cases to consider, depending on whether or not the $p$-stabilized form is new or old at $p$ (so $M=M_0p$ 
or $M=M_0$). 

Start with the case that $g_0$ is old at $p$. Then a further calculation (see Lemme 27 of \cite{pr89}) shows that 
$\{g, g\}_M  = (p-1)\cdot E_p\cdot \langle g_0, g_0\rangle_{M_0}$, where
$E_p = \pm p^{1-k/2} \alpha_p(1-\frac{p^{k-2}}{\alpha_p^2})(1-\frac{p^{k-1}}{\alpha_p^2})$, and 
$g_0$ is the newform of level $M_0$ associated
to $f$, and $\alpha_p$ is the unit root of the Hecke polynomial. The factor of $p-1$ comes from
the fact that the Petersson inner product in \cite{pr89}
is defined as an integral over a fundamental domain for $\Gamma_0$, not $\Gamma_1$. The  factors
$1-\frac{p^{k-2}}{\alpha_p^2}$ and $1-\frac{p^{k-1}}{\alpha_p^2}$ are units for $k>2$.  
When $k=2$, the term $1-1/\alpha_p^2$
may be a non-unit; this is so precisely when $g_0$ is congruent to a $p$-new form of of level $pM_0$.
The number  $1-1/\alpha_p^2$ is the relative congruence number of Ribet \cite{rib83}. 

In the $p$-old situation (and all weights), we define $$(g_0, g_0) = \frac{(g, g)_{M}}{1-1/\alpha_p^2}.$$
Note that $\alpha_p\neq \pm 1$, by the Weil bounds.
We remark that it can be shown that in fact $(g_0, g_0)$ as defined above coincides with the pairing of $g_0$ with 
itself defined via a Gorenstein pairing at level $M_0$ (as opposed to the $p$-stabilized level 
$M=M_0p$). We don't need this result, but mention it simply to justify the notation. 
We refer the reader to \cite{wil95}, Chapter 2, Section 2, for a full discussion of relative congruence numbers in weight 2.

If $f$ is new at $p$ (and hence of weight 2), then of course $\{g, g\}_M = E_p \langle g_0, g_0\rangle_M$
where $E_p=\pm 1$ is the eigenvalue of the Fricke involution.

The number ${\Omega_M}$ is almost the canonical period, but not quite: it 
depends not only on $g_0$, but on the choice of the prime $\pp$, and the stabilization of $g_0$
at the unit root $\alpha_p$ of the Hecke polynomial. We can get rid of this dependency as follows.

 If $f$ is old at $p$, we have 
$$\Omega_M = \frac{\{g, g\}_M}{( g, g)_M} = \text{unit}\cdot E_p \frac{\langle g_0, g_0\rangle_{M_0}}{( g, g)_{M}}=
\text{unit}\cdot p^{1-k/2}\frac{\langle g_0, g_0\rangle_{M_0}}{( g_0, g_0)_{M_0}}.
$$
Note that we have used Proposition 2.4 of \cite{wil95} to account for the non-unit congruence number for $g_0$ in the last step 
of equalities above.

If $f$ is new at $p$, so that $g=g_0$ and $M=M_0$, and we are in weight 2, we have 
$$
\label{periotd2} \Omega_M =  \frac{\{g, g\}_M}{( g, g)_M} =  E_p \frac{\langle g_0, g_0\rangle_{M_0}}{( g_0, g_0)_{M_0}}
=\text{unit} \cdot \frac{\langle g, g\rangle_M}{( g, g)_M}
$$

A common way of expressing the above formulae is by defining
\begin{equation}
\label{can-period}
\Omega_M = \text{unit}\cdot p^{1-k/2}\cdot \frac{\langle g_0, g_0\rangle_{M_0}}{( g_0, g_0)_{M_0}}.
\end{equation}
With this formulation, the only dependence on $\pp$ or $g$ is absorbed in the unit factor. Thus we may simply set 
\begin{equation}
\label{can-period-2}
\Omega_{g_0}^{\op{can}} = \frac{\langle g_0, g_0\rangle_{M_0}}{( g_0, g_0)_{M_0}},
\end{equation}
as stated in the introduction.
\end{Remark}

\begin{Remark} The  
factor $p^{1-k/2}$ which shows up in the comparison with $\Omega_M$
is important -- it shows up in the formulae   (\ref{petersson-formula}), and those of Schmidt in 
\cite{schmidt86}, \cite{schmidt88}, where it is simply carried 
around. As we shall see, it 
 exactly cancels unwanted powers of $p$ arising from (\ref{petersson-formula}). 

\end{Remark}

%
%
%

Now, if $h\in S_k(N,\oo)_\frm$ is arbitrary, then we have $\frac{(f, h)_N}{(f, f)_N}=
\frac{\{f, h\}_N}{\{f, f\}_N} = \frac{\{f, h\}_N}{(f, f)_N\Omega_{N}}$.
In view of the independence of the period on the level, we get the following key evaluation formula, valid for any $h\in S_k(N, \oo)_\frm$:

 \begin{Proposition} 
 \label{evaluation-formula} Assume that the Ihara hypotheses are valid. Then
we have the equalities $(f, h)_N = \frac{\{f, h\}_N}{\Omega_{N}} =\op{unit}\cdot \frac{\{f, h\}_N}{\Omega_{M}}$. 
Further, the quantity $$ \frac{\{f, h\}_N}{\Omega_{M}}=
\op{unit}\cdot  p^{k/2-1}\cdot 
 \frac{\{f, h\}_N\cdot (g_0, g_0)_{M_0}}{\langle g_0, g_0\rangle_{M_0}} $$ is $\pp$-integral.
 \end{Proposition}

\begin{Remark} Our next task will be to apply the machinery developed above to the case where $h=\mathcal{H}_{\chi}(n)$ is derived from a product 
of a theta
series and and Eisenstein series. However, there are two problems. First, this product is unlikely to be cuspidal, and second, it is not an element 
of $h\in S_k(N, \oo)_\frm$. Some care is  therefore required. The number $\{f, h\}_N = \langle f\circ W_N, h^c\rangle_N$ 
 makes sense for any $h\in M_k(N, \oo)\otimes\qbar$, since $f$ is cuspidal. Since the maximal ideal $\frm$ corresponding to $f$ is residually irreducible,
 we find that if $e_\frm$ is the idempotent in the Hecke algebra $\T=\oplus \T_{\frm_i}$ corresponding
to the maximal ideal $\frm$, then $h\circ e_\frm$ is cuspidal for any modular form $h$ of level $N$ and weight $k$. But now
 $f\circ e_\frm = f$, and the Hecke operators are self-adjoint under
 the modified pairing, hence $\{f, h\circ e_\frm \}_N=\{f, h\}_N$.  
 
Thus we may replace $h$ with $h\circ e_\frm$, and define $( f, h)_N = (f, h\circ e_\frm)_N$ for any $h\in M_k(N, \oo)\otimes\qbar$.
Then the same formalism as above applies. In particular, 
we still have $\frac{(f, h)_N}{(f, f)_N}=
\frac{\{f, h\}_N}{\{f, f\}_N} = \frac{\{f, h\}_N}{(f, f)_N\Omega_{N}}$, and the conclusion of Proposition \ref{evaluation-formula} applies without change.
\end{Remark}

\subsection{Integrality}
In view of the considerations above, we are led to compute the algebraic pairings $( f, \mathcal{H}_\chi(n))_N$ and $(f,  \mathcal{H}^m_\chi(n))_N$, with
$\mathcal{H}_\chi(n)$ and $\mathcal{H}_\chi^m(n)$ being as defined in (\ref{hchi}) and (\ref{hchim}). To use our previous
formulae, we need to assume $2\in S_0$. Then, we know already that 
$\mathcal{H}_\chi(n)$ and $\mathcal{H}_\chi^m(n)$ are integral, hence the corresponding pairings are integral as well. It remains only to
relate them to special values of $L$-functions. The starting point is Proposition \ref{evaluation-formula}, which reduces the calculation
to that of  analytic pairings $\{f, \mathcal{H}_\chi(n)\}_N$ and $\{f,  \mathcal{H}^m_\chi(n)\}_N$.  Pick any $m\geq m_\chi$.
%
The analytic pairing is computed in 
equation (\ref{petersson-formula-gamma1}), which states that
\begin{equation*}
\frac{\Gamma(n/2)}{(4\pi)^{-n/2}}p^{(2m_\chi-1)(k/2-1)} D_f(\chi,n) =   \op{unit}\cdot \alpha_p^{2(m_\chi-m)} \langle f\circ W_N, H_{\overline\chi}(n)\circ W_{N_\chi}\circ U_p^{2m-1}\circ T_\psi\rangle_{N}.
\end{equation*}

Using this formula, and plugging in all the definitions, we obtain
\begin{Corollary} 
\label{petersson-formula-2}
Suppose that $\eta$ is ramified and $m\geq m_\chi$ is any integer. 
There exists a $\pp$-adic unit $u$ depending only on $n$ and $k$ and $\epsilon$ and $\psi$ such 
that we have $$u\cdot \frac{p^{1-k/2}}{\pi^{n}}\left( \frac{p^{n-1}}{\psi(p)}\right)^{m_\chi} \left(\frac{1}{\alpha_p^2}\right)^{m_\chi-m}
\cdot g(\ol{\eta})\cdot D_f(\chi,n) =  
 \{ f, \mathcal{H}^m_\chi(n)\}_N.$$ where $\mathcal{H}_\chi^m(n) = \tilde{H}_\chi(n)^{\text{hol}}\circ U_p^{2m-1}\circ T_\psi\in M_k(N, \oo)\otimes\qbar$
 has $\p$-integral coefficients, and 
 $\tilde{H}_\chi(n)= \frac{\Gamma((n+1)/2)}{\pi^{(1+n)/2}} p^{m_\chi(3-2k+2n)/2} \cdot \frac{\sqrt{c_\psi p^{m_\chi}}}{g(\chi)} \cdot H_{\chi}(n)\circ W_{N_\chi}$.

\end{Corollary}

\begin{proof} This is a direct computation, using Lemmas, \ref{eisen-fourier} and \ref{theta-fourier} 
and applying the doubling formula for the $\Gamma$-function in the formula for the coefficients $d_{j,\nu}$; see also Lemma 4.2 of \cite{schmidt86}. 
One has also to use the factorization $g(\chi) = \psi(p^{m_\chi})\eta(c_\psi)g(\psi)g(\eta)$, as well as the normalization of the modified pairing 
\eqref{modified-petersson}.
The constant $u$ collects up all the various powers of $2, i$, and other quantities prime to $\pp$. 
\end{proof}

Observe that the formula above contains the nuisance factor $p^{k/2-1}$, which also appears in our period. 
Assuming Hypothesis \ref{hypothesis ihara}, so that $\Omega_{g_0}^{\op{can}}= \text{unit}\cdot p^{k/2-1}\Omega_N$, 
where $N$ is the level of the depleted form $f$,
and plugging in (\ref{can-period}), 
 we find that 
 \begin{align}
  \label{integrality-formula-1}
 (f, \mathcal{H}_\chi^m(n))_N  & =  u\cdot p^{1-k/2}\cdot \left( \frac{p^{n-1}}{\psi(p)}\right)^{m_\chi}  \left(\frac{1}{\alpha_p^2}\right)^{m_\chi-m}
\cdot \Gamma(n)\cdot G(\ol{\eta})\cdot \frac{D_f(\chi,n)}{\pi^n \Omega_N}  \\
& =u' \cdot \left( \frac{p^{n-1}}{\psi(p)}\right)^{m_\chi} \cdot  \left(\frac{1}{\alpha_p^2}\right)^{m_\chi-m}
\cdot \Gamma(n) \cdot G(\ol{\eta})\cdot \frac{D_f(\chi,n)}{\Omega_{g_0}^{\op{can}}}\\
& =u' \cdot \left( \frac{p^{n-1}}{\psi(p)}\right)^{m_\chi} \cdot  \left(\frac{1}{\alpha_p^2}\right)^{m_\chi-m}
\cdot \Gamma(n) \cdot G(\ol{\eta})\cdot\frac{(g_0, g_0)_{M_0}}{\pi^n \langle g_0, g_0\rangle_{M_0}}\cdot D_f(\chi,n).
\end{align}
Here $u'$ is some other unit, independent of $\chi$. 

Finally, we have to deal with $\mathcal{H}_\chi(n)\circ e$. It is not hard to see that the twisted trace operator $T_\psi$, which goes from level $N_\psi$ to
level $N$, commutes with the Hecke operator $U_p$, since $N_\psi/N$ has no common factor with $p$. There does not seem to be any particularly
pleasant way to deduce this fact from a classical perspective where the trace operator is given by matrices with rational integer entries,
but it is more or
less obvious from the point of view of representation theory, since the local trace involves primes away from $p$, while $U_p$ is concentrated at $p$. 
It is also evident it one considers modular forms as functions on test objects on moduli spaces of enhanced elliptic curves  -- $U_p$ is a sum over certain
subgroups of order $p$, while the trace from level $N_\psi$ involves subgroups of order prime to $p$. Anyway, we take this fact for granted, so that if $m$
is large, we have
$$\mathcal{H}_\chi^m(n) = \tilde{H}_\chi(n)^{\text{hol}}\circ U_p^{2m-1}\circ T_\psi =\tilde{H}_\chi(n)^{\text{hol}}\circ T_\psi\circ U_p^{2m-1}.$$
We may then consider an suitable increasing sequence of integers $m$, divisible by $p-1$, so that the forms $\mathcal{H}_\chi^m(n) \circ U_p$ converge to 
$\mathcal{H}_\chi(n)\circ e$.  The algebraic inner product is linear,
and we conclude that 
\begin{equation}
\label{integrality-formula-2}
 (f, \mathcal{H}_\chi(n)\circ e)_N  =\op{unit} \cdot \left( \frac{p^{n-1}}{\psi(p)\alpha_p^2}\right)^{m_\chi} 
\cdot \Gamma(n) \cdot G(\ol{\eta})\cdot\frac{(g_0, g_0)_{M_0}}{\pi^n \langle g_0, g_0\rangle_{M_0}}\cdot D_f(\chi,n).
\end{equation} In particular, the right-hand side is integral. Observe that the quantity on the right is a constant multiple of  
the one appearing in the definition of the
$\psi$-twisted $p$-adic L-function given in (\ref{interpolation}). The Euler factor disappears because $\eta$ is ramified at $p$.

\subsection{Level raising and congruences}
\label{congruence-section}
In view of the construction given above, it is more or less clear that that the $p$-adic L-functions satisfy
good congruences. To state the result, consider two $\mathfrak{p}$-ordinary and $\mathfrak{p}$-stabilized
newforms $g_1, g_2$, which are such that the residual 
representations $\overline{\rho}_{g_1}$ and $\overline\rho_{g_2}$ are isomorphic to some fixed representation
$\ol\rho$. We assume, as always, that $\ol\rho$ is irreducible, ordinary, and distinguished. 
We write $M_i$ for the level of $g_i$, and $\epsilon_i$ for the corresponding nebentype character. 
The considerations in the previous paragraphs give an integral construction of imprimitive $p$-adic L-functions
for each $g_i$, associated to some choice of a set $S_0$ of primes $q\neq p, 2\in S_0$, which depends on $i$.
We now show that we can choose a set for $S_0$ such that depletion of $g_1$ and $g_2$ yields
 a common level $N$ where semisimplicity is retained and we can apply our construction.

Recall that $\ol\rho$ denotes the common  2-dimensional residual representation for the $g_i$. 
Let $q$ denote a prime number, $q\neq p$. In line with our previous convention, we shall say that $\ol\rho$ is ordinary
at $q$ if the subspace $\ol\rho^{I_q}$ of invariants in $\ol\rho$ under an inertia group $I_q$ at $q$ has dimension $1$. If
$\ol\rho^{I_q}=0$, we say that $\ol\rho$ is depleted at $q$, and if $\ol\rho^{I_q}=\ol\rho$, we say that $\ol\rho$ is unramfied.
Since each $g_i$ is a lift of $\ol\rho$, it follows that each $\rho_{g_i}$ is depleted at $q$ if the common representation
$\ol\rho$ is so. If $\ol\rho$ is ordinary, then each $\rho_{g_i}$ is either depleted or ordinary. Let $\ol{M}$
denote the conductor of $\ol\rho$, as defined on page 104 of \cite{gouvea90}.

Our first task is to determine the possible difference between the levels $M_1$ and $M_2$. The answer is given by Proposition 6
of \cite{gouvea90}, which we paraphrase as follows:

\begin{Lemma}
\label{gouvea90} For each prime $q$ we have $\ord_q(M_1)=\ord_q(M_2)=\ord_q(\ol{M})$, unless one of the following occurs:
\begin{itemize}
\item Either $\ol\rho$ is unramfied, in which  case, $\ord_q(M_i)\leq 2$, for $i=1,2$; or 
\item $\ol\rho$ is ordinary, in which case $\ord_q(M_i)\leq \ord_q(\ol{M})+1$, for $i=1, 2$. If
 $\ord_q(M_i)=  \ord_q(\ol{M})+1$, then $\rho_{g_i}$ is depleted at $q$. If
 $\ord_q(M_i)=  \ord_q(\ol{M})$, then $\rho_{g_i}$ is ordinary.
 \end{itemize}
 \end{Lemma}

It is therefore clear that if we take $S_0$ to be any finite set primes containing $2$, as well as 
 all primes dividing $M_1M_2$, plus any any finite set of primes away from the levels,
 then the considerations of the previous 
paragraphs apply. The corresponding depletions have the same level $N$, where
$\op{ord}_q(N)=\op{ord}_q\ol{M}$, except at primes
$q\in S_0$ where $\ol{M}$ is unramified, in which case $\ord_q(N)=2$, and primes $q\in S_0$ where $\ol{M}$ is ordinary, 
in which case $\ord_q(N)=\ord_q(\ol{M})+1$. 

 We can now state the theorems around congruences for the imprimitive symmetric square L-function, but we need to 
 recall all the hypotheses and notation.
Thus, suppose that $\pp$ is a prime
 of $\qbar$ with residue characteristic $p\geq 5$, and that $g_1, g_2$ are $p$-stabilized and $\p$-ordinary
 newforms of weight $k$, and 
 level $M_1, M_2$ respectively, such that  the the Fourier
 coefficients $a(q, g_i)$ satisfy the congruence $a(q, g_1)\equiv a(q, g_2)\pmod{\pp}$, for each prime 
 $q\nmid M_1M_2p$. Assume that the corresponding residual representations are absolutely irreducible and $p$-distinguished,
and the the nebentype characters are trivial on $(\Z/p\Z)^*$. 
Let $\alpha_{i,p}$ denote the eigenvalue under $U_p$ of $g_i$.
  Let $\psi$ denote an even character of conductor prime to $p$, such that $(\psi\epsilon_i)^2\not\equiv 1\pmod{\pp}$.
  and let $\eta$ denote a nontrivial
 Dirichlet character of $p$-power conductor, and set $\chi=\psi\eta$.
 Let $\Omega_1, \Omega_2$ denote the canonical periods associated
 to  $g_1, g_2$ and the prime $\pp$, as above. Thus 
 $\Omega_i = p^{k/2-1} \frac{\langle g_{0,i}, g_{0,i} \rangle_{M_{0,i}}}{(g_{0,i}, g_{0,i})_{M_{0,i}}}$,
 where $g_{0, i}$ is the newform of level $M_{0,i}$ corresponding to $g_i$. 
  Let $n$ denote an odd integer in the range $1\leq n\leq k$. Finally, fix a finite set $S,_0$ 
of primes $q\neq p$ containing $2$ and all primes dividing $M_1M_2$. 
  Let $f_1, f_2$ denote the depleted forms of level $N$, associated to the forms $g_1, g_2$, and the set
  $S_0$. 
    
 \begin{Th}\label{special values congruence} Let the hypotheses  and notation be as above. 
 Then there exist units $u_i$, independent of $\chi$, such that we have the congruence 
 \begin{equation}
 \begin{split}
 u_1 \cdot \left( \frac{p^{n-1}}{\psi(p)\alpha_{1,p}^2}\right)^{m_\chi} 
\Gamma(n)G(\ol{\eta})\cdot \frac{D_{f_1}(\chi,n)}{\pi^n\Omega_1}\\ \equiv
u_2 \cdot \left( \frac{p^{n-1}}{\psi(p)\alpha_{2,p}^2}\right)^{m_\chi} 
 \Gamma(n) G(\ol{\eta})\frac{ D_{f_2}(\chi,n)}{\pi^n\Omega_2} \pmod{\pp}.
 \end{split}
\end{equation}
  \end{Th}
 
 \begin{proof} This follows from the linearity of the functional  $S_k(N, \oo)\rightarrow \oo $ given by $ x \mapsto 
 ( x, \mathcal{H}_\chi(n)\circ e)_N$. 
 \end{proof}

\subsection{The primitive L-function and $p$-adic interpolation}
\label{primitive-ss-Lfunction}
We now write down the relationships between the primitive and variously imprimitive L-functions, and 
the interpolation properties that characterize the $p$-adic L-functions. We also have to remove
 the hypothesis $2\in S_0$ which we imposed for the purposes of our calculation of special values. 
For notational
simplicity, let us fix the
newform $g_0$, and write the level of $g_0$ as $M_0$. The corresponding $p$-stabilized newform will be denoted by $g$
and its level shall be denoted by $M$.
For each prime $q$, we have a complex
representation $\pi_q$ of $\op{GL}_2(\Q_q)$ 
associated to $g$. 
The first task is to work out the Euler factors of 
the symmetric square lift $\Pi_q$ of $\pi_q$ to $\op{GL}_3$. This is all contained in \cite{GJ78}, 
and is recapitulated in Section 1 of \cite{schmidt88},
especially Lemmas 1.5 and 1.6,
but some translation is required. We notice first of all
that the representations $\Pi$ and $\Sigma$ considered by Schmidt
are not exactly the symmetric square of \cite{lz16}, and that his 
normalization introduces an inverse when comparing
with the Euler product of Shimura considered here. The exact 
relationship is given in the last line of page 603 of \cite{schmidt88}. For us, the
point is that the Euler factors of our $D_g(\chi_0, s)$ coincide
with Schmidt's $L(s-k+1, \Sigma\otimes\chi^{-1})$ for any 
primitive (in our case even) Dirichlet
character $\chi_0$ with corresponding idele class character $\chi$,
at almost all primes. With this normalization in mind, one can
read off the Euler factors for the automorphic
representation $\Pi$ from Lemmas 1.5 and 1.6 of \cite{schmidt88}. 

To state 
the result, let us write $${L}(r_{g}\otimes\chi, s)=\prod_q
P_q(\chi, q^{-s})^{-1}$$ to denote the complex L-function
associated to the Galois representation $r_{g}\otimes\chi$, as in the introduction. 
Now let $S_0$ be any finite set of primes $q\neq p$. We no longer require
$2\in S_0$. 
Then the imprimitive $L$-function is defined by the same formula, except with the product
taken over primes $q\notin S_0$. 
Recall also that we have are identifying Galois characters and Dirichlet characters
via $\chi(\frob(q))=\o\chi(q)$, and $\frob(q)$ is normalized so that $\frob(q)$ lifts $x\mapsto x^q$.

For all but finitely primes $q$, we have 
$$P_q(\chi, q^{-s})= \left( (1-\chi(q)\alpha_q\beta_q q^{-s})(1-\chi(q)\beta_q^2q^{-s})(1-\chi(q)\alpha_q^2q^{-s})\right).$$
If the set $S_0$ is sufficiently large, and $\chi$ is ramified at $p$, then $D_f(\chi, s) = L_{S_0}(r_g\otimes\chi, s)$, and so we get
$$D_f(\chi, s) = D_g(\chi, s)\cdot \prod_{q} P'_q(\chi, q^{-s}) = L(r_{g}\otimes\chi, s)\cdot \prod_{q} P_q(\chi, q^{-s})=L_{S_0}(r_g\otimes\chi, s).$$
where once again the $P'_q$ and $P_q$ are polynomials in the variables $X=q^{-s}$
 (compare \cite{lz16}, Proposition 2.1.5). The polynomials $P_q$ are those given by Schmidt, whereas the $P'_q$ are the ones
 given by the Euler factor at $q\in S_0$ in Shimura's Euler product. 
 Each of the products is taken over the finite set of primes
in $S_0$, together with the primes of ramification of $r_g\otimes\chi$, with the understanding that some of the factors may be trivial.

We may view the polynomials $P_q$ (for any character) as elements of the Iwasawa 
algebra $\Lambda = \cO[[\text{Gal}(\Q_{\op{cyc}}/\Q)]]$
by replacing $X$ with a Frobenius element at $q$. In particular, we may do so
for the polynomials $P_q = P_q(\psi_t)$, where $\psi$ is our fixed even character of conductor prime to $p$.
 We remark that, under the identification of $\cO[[\text{Gal}(\Q_{\op{cyc}}/\Q)]]$
with the subgroup of $1+p\zz_p\subset \zz_p^\times$, we have $\frob(q) = q_w = \eta_1^{-1}(q)q$, as in the introduction. 

In the setup of $p$-adic L-functions, we have $\chi = \psi\eta$, where 
$\eta$ has $p$-power 
conductor. Furthermore, we have $\psi\eta =\psi_t\eta_w$, for some even $t, 0\leq t\leq p-2$, and some character $\eta_w$
of $p$-power order. A glance at the formulae on pages 604--605 in \cite{schmidt88} shows that in fact
one has 
\begin{equation}
\label{local-factor}
\eta_{w, n}(P_q(\psi_t)) = P_q(\psi_{t+n-1}\eta_w, q^{-n}),
\end{equation}
consistent with the basic formula (\ref{interpolation}).

Following \cite{lz16}, the primitive $p$-adic L-function associated to the newform $g_0$ of level $M$ 
and the representation
$r_g\otimes\psi_t$ is an element $\mathscr{L}^{\op{an}}(r_g\otimes\psi_t)$ of $\Lambda$
characterized by 
\begin{equation}
\label{messy-definition}
\eta_{w, n}(\mathscr{L}^{\op{an}}(r_g\otimes\psi_t)) = \Gamma(n)
\cdot E_p(n, \chi)\cdot G(\eta_{1-t-n}\eta_w^{-1})\cdot \frac{{L}(r_g\otimes \psi_{t+n-1}\eta_w, n)}{\pi^{n}\Omega^{\op{can}}_{g_0}}.
\end{equation}
for $n$ odd, $1\leq n\leq k-1$. The Euler factor
$E_p(n, \chi)$ is given by
$$E_p(n,\chi) = (p^{n-1}\psi(p)^{-1}\alpha_p^{-2})^{m_\chi}$$
if $\eta$ is nontrivial and has conductor $p^{m_\chi}> 1$. If $\eta$ is trivial and $g_0$ has level prime to $p$,
then 
$$E_p(n, \eta) = (1- p^{n-1}\psi(p)^{-1}\alpha_p^{-2})(1-\psi(p)p^{k-1-n})
(1-\psi(p)\beta_p^2p^{-n}).$$
A similar formula holds when $k=2$ and $g$ has level divisible by the first power of $p$; we omit it here, as we
do not need it. We remark
also that the construction of this paper does not prove that the $p$-adic L-function exists.

Observe
that our formula above is the same as that in \cite{lz16}, except that 
\begin{itemize}
\item We have used the canonical period rather than the Petersson product.
\item We have suppressed unit factors of $2$ and $i$ that depend only on the weight.
\item We have used our parity assumptions to get rid of various minus signs.
\item We have adjusted the action of the Iwasawa algebra to match the Selmer group
defined by Greenberg.
\end{itemize}

It is clear
that if such a function exists, then it is characterized by the validity
of the formula above, for any infinite collection of characters of the form
$\eta_{w,n}$. If $S_0$ denote any finite set of prime numbers $q\neq p$ (which may not contain $q=2$),
then $\mathscr{L}_{S_0}^{\op{an}}(r_g\otimes\psi_t)$ (assuming it exists)
is an element of $\Lambda$ 
characterized by the analogue of (\ref{messy-definition}).
%

The existence of an L-function interpolating the values of $D_g(\chi, s)$ was proven
 by Schmidt \cite{schmidt86}, \cite{schmidt88}.  
He states in his work that very similar results were obtained by Hida, but never
published. He (Schmidt) subsequently proved the existence of the primitive 
L-function $\mathscr{L}^{\op{an}}(r_g\otimes\psi_t)$ under some hypotheses, which were later removed by Hida \cite{hid90} and
Dabrowski-Delbourgo \cite{dd}. Schmidt did not 
construct an interpolation of the imprimitive $D_f(\chi, s)$, or the imprimitive $\mathscr{L}^{\op{an}}_{S_0}(r_g\otimes\psi_t)$,
but in fact those follow easily:
one simply multiplies the existing L-function by the appropriate Euler factors, 
each of which is represented by an element $P_1(X)$ or $P'_q(X)$ of $\Lambda$. 

As we have already remarked, this construction of imprimitive $p$-adic L-functions does not help in proving congruences,
since we cannot compare the L-functions for different forms in any way: each one is obtained from
a construction at a different minimal level, which are not related in any simple way.

In any case, we see from equation (\ref{local-factor}) that the relationship between
the primitive and imprimitive $p$-adic L-functions is given by
\begin{equation}
\label{prim-imprim}
\mathscr{L}^{\op{an}}_{S_0}(r_g\otimes\psi_t) = \prod_{q\in S_0} P_q(\psi_t)\cdot \mathscr{L}^{\op{an}}(r_g\otimes\psi_t).
\end{equation}
A priori, both L-functions above are elements of $\Lambda\otimes\mathbb{Q}$. A similar formula of course holds
for the $p$-adic L-function interpolating $D_f(\chi, s)$ and $D_g(\chi, s)$. 

We can now give the proof of the various remaining results on $p$-adic L-functions
stated in the introduction. We start with a simple lemma.
\begin{Lemma}
For any prime $q\in S_0$, the elements $P_q, P'_q \in\Lambda$ have $\mu$-invariant zero. 
\end{Lemma}

\begin{proof} The statement for $P_q$ follows from the explicit
formulae for the Euler factors in \cite{schmidt88}, or the observation there on page 605
that the polynomials $P_q$ all satisfy $P(0)=1$. The same argument works for $P'_q$, since these
are given explicitly in Shimura's defining Euler product.
\end{proof}

\begin{Corollary} We have $\mu_{S_0}^{\op{an}}=0\iff \mu^{\op{an}}=0$.
\end{Corollary}

This lemma implies Proposition \ref{analytic-invariants-intro},
simply by taking $g=g_i$, and $\sigma^{(q)}_i$ to be the degree of the polynomial $P_q$ 
associated above to $g_i$ at $q$, and using the fact that the 
$\mu$-invariant of $P_q$ is zero
in the formula (\ref{prim-imprim}).

Next we deal with integrality properties, as stated in Theorem \ref{integrality-thm-intro}.
 We claim that in fact both $\mathscr{L}^{\op{an}}_{S_0}(r_g\otimes\psi) $ and $\mathscr{L}^{\op{an}}(r_g\otimes\psi)$
lie in $\Lambda$ and are integral, with the same canonical period. 
First we apply our construction to the $p$-stabilized newform $g$ and large set $S_0$ 
containing $2$ and all the bad primes. 
Then since $D_f(\chi, s) = L_{S_0}(r_g\otimes\chi, s)$ for sufficiently large $S_0$,
it follows from the formulae (\ref{integrality-formula-1})
and (\ref{integrality-formula-2}) that the quantity in 
the interpolation formula  (\ref{interpolation}) is integral, for almost all characters $\eta_n$. 
The 
Weierstrass preparation theorem, applied
to $\mathscr{L}^{\op{an}}_{S_0}(r_g\otimes\psi)$, shows that the latter is an element
of $\Lambda$. But now we have 
$\mathscr{L}^{\op{an}}_{S_0}(r_g\otimes\psi) = \mathscr{L}^{\op{an}}(r_g\otimes\psi) \cdot \prod P_q$,
for integral polynomials $P_q$ with $\mu$-invariant zero, so $\mathscr{L}^{\op{an}}(r_g\otimes\psi) $
is integral because the ring of power series with coefficients in a field is an integral domain. 
We may now obtain the result for any given set $S_0$, simply by multiplying by the factors $P_q(X), q\in S_0$. 
This proves Theorem \ref{integrality-thm-intro}.

Finally, we have to deal with congruences. Let $g_1, g_2$ be $p$-congruent newforms satisfying our running conditions.
Let $S$ denote
any set of primes including $2$, and the set of primes dividing $M_1M_2$, and let
$S_0=S\backslash\{p\}$ be as above.  Furthemore, we have to assume that the Ihara hypotheses hold. We claim that we have

\begin{Proposition}\label{p-adic LFs congruent} Let the notation be as above. Then we have
 $\mathscr{L}_{S_0}^{\op{an}}(r_{g_1}\otimes\psi_t) \equiv  u 
 \mathscr{L}_{S_0}^{\op{an}}(r_{g_2}\otimes\psi_t)\pmod{\pp}$, where $u$ is a $p$-adic unit and the congruence
 is that of elements in the completed group algebra $\oo[[\zz_p]]]$. 
 \end{Proposition}
 
 \begin{proof}  
This follows from the congruence of special values in Theorem \ref{special values congruence}, and the Weierstrass preparation theorem.
   \end{proof}

We remark that the result in Theorem \ref{special values congruence}, and the statement of the theorem above,
 remain valid without the Ihara hypothesis, if replaces the canonical periods with the periods $\Omega_{g_i, N}$. 
However, $N$ depends on the set $S_0$. 

Finally, we observe that the analytic part of Theorem \ref{intro-thm} follows immediately, since two congruent Iwasawa functions with 
$\mu$-invariant zero necessarily have the same $\lambda$-invariant, and that if one has $\mu$-invariant zero, then so does the other.

\section{Imprimitive Iwasawa Invariants: the algebraic side}\label{s 3} We start by recalling the notation. From the previous
sections. Throughout, let $p\geq 5$ be a fixed prime and $g$ be a normalized Hecke-eigencuspform of weight $k\geq 2$ on the congruence group $\Gamma_0(M)$. Denote the number field generated by the field of Fourier coefficients of $g$ by $L$. For each prime $q$, choose an embedding $\iota_q:\bar{\Q}\hookrightarrow \bar{\Q}_q$. Let $\p|p$ be the prime of $L$ such that the inclusion of $L$ in $L_\p$ is compatible with $\iota_p$. Denote by $K$ the completion of $L$ at $\p$, and $\mathcal{O}$ the valuation ring of $K$ with uniformizer $\varpi$. Associated with $g$ is the continuous Galois representation $\rho_g:\op{Gal}(\bar{\Q}/\Q)\rightarrow \op{GL}_2(K)$. Let $\op{V}_g\simeq K^2$ be the underlying $2$-dimensional vector space on which $\op{Gal}(\bar{\Q}/\Q)$ acts via $K$-linear automorphisms. Fix a Galois stable $\cO$-lattice $\op{T}_g$ inside $\op{V}_g$. Let $\F$ be the residue field of $\cO$. The mod-$\p$ reduction of $\rho_g$ is denoted by
\[\bar{\rho}_g:\op{G}_{\Q}\rightarrow \op{GL}_2(\F),\]and it follows from the Brauer-Nesbitt theorem that the semi-simplification of $\bar{\rho}_g$ is independent of the choice of lattice $\op{T}_g$.
Throughout, we make the following assumptions on $g$:
\begin{enumerate}
    \item $g$ is ordinary and $p$-distinguished at $\p$,
    \item $\bar{\rho}_g$ is absolutely irreducible.
\end{enumerate}
Since $\bar{\rho}_g$ is absolutely irreducible, the choice of Galois stable lattice $\op{T}_g$ is unique.
Letting $\op{G}_q$ denote $\op{Gal}(\bar{\Q}_q/\Q_q)$, we note that the choice of embedding $\iota_q$ gives an inclusion of $\op{G}_q$ into $\op{G}_{\Q}$. Let $\chi_{\op{cyc}}:\op{G}_p\rightarrow \cO^\times$ denote the $p$-adic cyclotomic character. Since $g$ is ordinary at $\p$, there is a short exact sequence 
\[0\rightarrow \op{T}_g^+\rightarrow \op{T}_g\rightarrow \op{T}_g^-\rightarrow 0\] of $\op{G}_p$-stable $\mathcal{O}$-lattices such that there are unramified characters $\gamma_1, \gamma_2:\op{G}_{p}\rightarrow \mathcal{O}^{\times}$ for which
\[\op{T}_g^+\simeq \mathcal{O}(\chi_{\op{cyc}}^{k-1} \gamma_1) \text{ and } \op{T}_g^-\simeq \mathcal{O}(\gamma_2).\]Fix a finite order even character $\psi$ of conductor $c_\psi$ coprime $Mp$. Let $t$ be an even integer in the range $0\leq t\leq p-2$ and recall that $\psi_t$ denotes the Dirichlet character $\psi\eta_1^t$, defined in the introduction. Consider the lattice $\bfT_g:=\op{Sym}^2 \op{T}_g$ and the symmetric square representation 
\[r_g\otimes \psi_t:=\op{Sym}^2(\rho_g)\otimes \psi_t:\op{Gal}(\bar{\Q}/\Q)\rightarrow \op{GL}_3(\mathcal{O}).\]
Set $\bfV_g:=\bfT_g\otimes \Q_p$ and $\bfA_g:=\bfV_g/\bfT_g$.
The representation $\bfT_g$ is $\p$-ordinary, i.e., is equipped with a filtration of $\op{G}_p$-modules
\[\bfT_{g}=\mathcal{F}^0(\bfT_g)\supset \mathcal{F}^1(\bfT_g)\supset \mathcal{F}^2(\bfT_g)\supset \mathcal{F}^3(\bfT_g)=0.\] For $j=1,2$, and unramified characters $\delta_j$, we have that \[\begin{split}&\op{gr}_0(\bfT_g)\simeq \mathcal{O}(\chi_{\op{cyc}}^{2k-2}\delta_0),\\
&\op{gr}_1(\bfT_g)\simeq \mathcal{O}(\chi_{\op{cyc}}^{k-1}\delta_1),\\
& \op{gr}_2(\bfT_g)\simeq \mathcal{O}(\delta_2).
\end{split}\]

\par With this notation in place, we consider Hecke $\pp$-stabilized eigencuspforms $g_i$ of weight $k\geq 2$, level $M_i$,
and character $\epsilon_i$, as in the introduction.  
 Setting $L$ to be the number field generated by the Fourier coefficients of $g_1$ and $g_2$, let $\mathfrak{p}$ be the prime of $L$ above $p$ corresponding to the choice of $\iota_p$.  Let $\psi$ denote a Dirichlet character of conductor coprime to $ p$ such that 
 $(\psi\epsilon_i)^2\not\equiv 1\pmod{\pp}$.
 Assume that $\bar{\rho}_{g_i}$ is absolutely irreducible and that the following equivalent conditions are satisfied.
\begin{enumerate}
    \item The residual representations are isomorphic: $\bar{\rho}_{g_1}\simeq \bar{\rho}_{g_2}$.
    \item For all primes $q\neq p$ coprime to the level of $g_1$ and $g_2$, the Fourier coefficients satisfy the congruence
    \[a(q,g_1)\equiv a(q,g_2)\mod{\varpi}.\]
\end{enumerate}
Note that $\bfT_{g_i}$ fits into a short exact sequence 
\[0\rightarrow \bfT_{g_i}^+\rightarrow \bfT_{g_i}\rightarrow \bfT_{g_i}^-\rightarrow 0,\]where $\bfT_{g_i}^+=\mathcal{F}^1(\bfT_{g_i})$ and $\bfT_{g_i}^-=\bfT_{g_i}/\bfT_{g_i}^+$. Set $\bfA_i$ (resp. $\bfA_i^{\pm}$) to denote the $p$-divisible Galois module $\bfT_{g_i}\otimes \Q_p/\Z_p$ (resp. $\bfT_{g_i}^{\pm}\otimes \Q_p/\Z_p$). Note that $\bfA_i\simeq (K/\mathcal{O})^d$, where $d=3$. Let $d^{\pm}$ be the dimensions (over $K$) of the $\pm$ eigenspaces for complex conjugation on $\bfV_{g_i}$, we have that $d^+=2$ and $d^-=1$ and that $\bfA_i^{\pm}\simeq (K/\mathcal{O})^{d^{\pm}}$.

\par Let $\Q_n$ be the subfield of $\Q(\mu_{p^{n+1}})$ degree $p^n$ and set $\Q_{\op{cyc}}:=\bigcup_{n\geq 0} \Q_n$. Letting $\Gamma:=\op{Gal}(\Q_{\op{cyc}}/\Q)$, we fix an isomorphism $\op{Gal}(\Q_{\op{cyc}}/\Q)\xrightarrow{\sim} \Z_p$. The extension $\Q_{\op{cyc}}$ is the cyclotomic $\Z_p$-extension of $\Q$. The Iwasawa algebra $\Lambda$ is defined as the following inverse limit $\Lambda:=\varprojlim_n \Z_p[\op{Gal}(\Q_n/\Q)]$, and is isomorphic to the formal power series ring $\Z_p\llbracket T\rrbracket$. We fix
a finite set $S$ of primes $q$ including all those that divide $c_\psi M_1M_2p$, and we let $\Q_S$ denote the maximal algebraic extension of $\Q$ unramified outside the set of primes $S$ and the infinite places. Further, define $\bfA_{i,\psi_t}:=\bfA_i\otimes \psi_t$
for an even integer $t$. 
Then the $p$-primary Selmer group $\Sel(\bfA_{i,\psi_t}/\Qcyc)$ is defined as the kernel of the following restriction maps
\[
\lambda_i:H^1\left(\Q_{S}/\Qcyc, \bfA_{i,\psi_t}\right)\rightarrow \bigoplus_{q\in S}\cH_q(\bfA_{i,\psi_t}/\Qcyc).
\]
Here for each prime $q\neq p$, the local term is defined as follows
\[
\cH_q(\bfA_{i,\psi_t}/\Qcyc) = \bigoplus_{\eta|q} H^1( {\Q}_{\op{cyc},\eta}, \bfA_{i,\psi_t}),
\]
where $\Q_{\op{cyc}, \eta}$ is the union of all completions of number fields contained in $\Qcyc$ at the prime $\eta$. Note that since all primes are finitely decomposed in $\Qcyc$, the above direct sum is finite.
The definition at the prime $q=p$ is more subtle, set
\[
\cH_p(\bfA_{i,\psi_t}/\Qcyc) = H^1( \Q_{\op{cyc}, \eta_p}, \bfA_{i,\psi_t})/\mathcal{L}_{\eta_p}
\]
with 
\[
\mathcal{L}_{\eta_p} = \ker\left( H^1( \Q_{\op{cyc}, \eta_p}, \bfA_{i,\psi_t}) \rightarrow H^1( I_{\eta_p}, \bfA_{i,\psi_t}^{-})\right).
\]
Here $\eta_p$ is the unique prime of $\Q_{\op{cyc}}$ above $p$, and $I_{\eta_p}$ denotes the inertia group at $\eta_p$. 
Then we have the following conjecture of Coates-Schmidt and Greenberg.
\begin{Conjecture}[\cite{CS}, \cite{Gre89}] Assume that $\psi_t$ is even. 
Then $\Sel(\bfA_{i,\psi_t}/\Q_{\op{cyc}})$ is a cotorsion $\Lambda$-module.
\end{Conjecture}
This conjecture has been settled by Loeffler and Zerbes in many cases (cf. \cite{lz16}). 

\begin{Lemma} Let $W=\bfA_{i,\psi_t}[\p]$ and let $W^* = \op{Hom}(W, \mu_{p^{\infty}})$ denote the Tate dual. 
Then we have $H^0(\Q, W)=H^0(\Q, W^*)=0$.
\end{Lemma}

\begin{proof} This follows from our assumption that $(\psi\epsilon_i)^2\not\equiv 1\pmod{\pp}$. Indeed, since $V = \overline\rho_{g_i}$ is irreducible, we have $V\otimes \det^{-1}\cong \hom(V, \Q_p/\Z_p)$ and
$\op{Ad}(V)=\hom(V, V)\cong V\otimes \hom(V, \Q_p/\Z_p)\cong V\otimes V\otimes \det^{-1}$, so that 
  $\op{Ad}^0(V)\cong\sym(V)\otimes\det^{-1}
\cong W\otimes\psi_t^{-1}\epsilon_i^{-1}\eta_{1-k} $. So 
we find that  if $H^0(\Q,W)\neq 0$, then $\op{Ad}^0(V)$ contains a line on which $\gal(\overline{\Q}/\Q)$ acts
via the character $\mu=(\psi\epsilon_i\eta_{k-1+t})^{-1}$. Since $V$ is irreducible, it is easy to see that that $V\otimes\mu
\cong V$. By taking determinants of both sides,  we conclude that
 $\mu=\psi\epsilon_i\eta_{k-1+t}$ must be  quadratic. But each of $\psi$ and $\epsilon$ is unramified at $p$,
whilst the Teichm\"uller character is unramified outside $p$, so we find that $(\psi\epsilon)^2\equiv 1\pmod{\p}$,
 which is a contradiction.  The proof for $W^*$ is analogous, using the fact that $(\ad^0)^*\cong \ad^0(1)$. 
 \end{proof}

\begin{Proposition}Let $\bfA_{i,\psi_t}$ be as above. Assume 
that $\psi_t$ is even. 
Then the localization map $\lambda_i$ is surjective.
\end{Proposition}
\begin{proof}
We let $\bfT_{i,\psi_t}^*:=\op{Hom}(\mathbf{T}_{g_i}\otimes \psi_t, \mu_{p^{\infty}})$. Then we have that $H^0(\Q, \bfT_{i,\psi_t}^*)=0$. Since $\op{Gal}(\Q_{\op{cyc}}/\Q)$ is pro-$p$, it follows that $H^0(\Q_{\op{cyc}}, \bfT_{i,\psi_t}^*)=0$ as well and thus in particular, is finite. The result follows from \cite[Proposition 2.1]{GV00}.
\end{proof}

Denote by $S_0:=S\backslash \{p\}$ and introduce the $S_0$-imprimitive Selmer group to be the Selmer group obtained by imposing conditions only at $p$.

\begin{Definition}
The \emph{imprimitive Selmer group} is defined by:
\[
\Sel^{S_0}(\bfA_{i,\psi_t}/\Qcyc) = \ker\left( H^1\left( \Q_{S}/\Qcyc, \bfA_{i,\psi_t}\right)\xrightarrow{\theta_0} \mathcal{H}_p(\bfA_{i,\psi_t}/\Qcyc)\right).
\]
\end{Definition}
Since the map defining the Selmer group is surjective, it follows that
\begin{equation}
\label{quotient}
\Sel^{S_0}(\bfA_{i,\psi_t}/\Qcyc)/\Sel(\bfA_{i,\psi_t}/\Qcyc)\simeq \bigoplus_{q\in S_0}\cH_q(\bfA_{i,\psi_t}/\Qcyc).
\end{equation}
\begin{Lemma}\label{local mu is 0}
If $q\neq p$, then $\cH_q(\bfA_{i,\psi_t}/\Qcyc)$ is a cofinitely generated and cotorsion $\Lambda$-module with $\mu$-invariant equal to $0$. 
\end{Lemma}
\begin{proof}
It suffices to show that $\cH_q(\bfA_{i,\psi_t}/\Qcyc)$ is a cofinitely generated as a $\Z_p$-module, or equivalently, the $p$-torsion subgroup $\cH_q(\bfA_{i,\psi_t}/\Qcyc)[\p]$ is finite. Consider the short exact sequence
\[0\rightarrow \bigoplus_{\eta|q} \frac{H^0(\Q_{\op{cyc}, \eta}, \bfA_{i,\psi_t})}{\p\left( H^0(\Q_{\op{cyc}, \eta}, \bfA_{i,\psi_t})\right)}\rightarrow \bigoplus_{\eta|q} H^1( \Q_{\op{cyc}, \eta}, \bfA_{i,\psi_t}[\p]) \rightarrow \cH_q(\bfA_{i,\psi_t}/\Qcyc)[\p]\rightarrow 0.\]
The set of primes $\eta|q$ of $\Qcyc$ is finite and so is $ H^1( \Q_{\op{cyc}, \eta}, \bfA_{i,\psi_t}[\p])$. The result follows.
\end{proof}
Let $\sigma_i^{(q)}$ denote the $\Z_p$-corank of $\cH_q(\bfA_{i,\psi_t}/\Qcyc)$ for $q\in S_0$.
Set $\lambda^{S_0}(\bfA_{i,\psi_t}/\Qcyc)$ to be the $\lambda$-invariant of $\Sel^{S_0}(\bfA_{i,\psi_t}/\Qcyc)$. It follows from the structure theory of $\Lambda$-modules that 
\[\lambda^{S_0}(\bfA_{i,\psi_t}/\Qcyc)=\op{corank}_{\Z_p} \left(\Sel^{S_0}(\bfA_{i,\psi_t}/\Qcyc)\right).\]
It follows from $\eqref{quotient}$ that the following relation is satisfied:
\begin{equation}
\label{relating im primitive and classical lambda invariant}
\lambda^{S_0}(\bfA_{i,\psi_t}/\Qcyc)=
\lambda(\bfA_{i,\psi_t}/\Qcyc) + \sum_{q\in S_0} \sigma_i^{(q)}.
\end{equation}
Analogous to the classical and imprimitive Selmer group, we also define the \emph{reduced} classical and imprimitive Selmer groups which we denote by $\Sel(\bfA_{i,\psi_t}[\p]/\Qcyc)$ and $\Sel^{S_0}(\bfA_{i,\psi_t}[\p]/\Qcyc)$, respectively. 
\par For $q\in S_0$ set
\[
\cH_q(\bfA_{i,\psi_t}[\p]/\Qcyc) := \prod_{q^\prime|q} H^1\left( {\Qcyc}_{,q^\prime}, \bfA_{i,\psi_t}[\p]\right),
\]
and for $q=p$, set
\[\cH_q(\bfA_{i,\psi_t}[\p]/\Qcyc):= H^1(\Q_{\op{cyc}, \eta_p}, \bfA_{i,\psi_t}[\p])/\overline{\mathcal{L}}_{\eta_p},\] where 
\[\overline{\mathcal{L}}_{\eta_p}:=\ker\left( H^1\left( {\Q}_{\op{cyc},\eta_p}, \bfA_{i,\psi_t}[\p]\right) \rightarrow H^1\left( I_{\eta_p}, \bfA_{i,\psi_t}^-[\p]\right)\right).\]
\begin{Definition}
The reduced imprimitive Selmer group is defined as follows
\[\Sel^{S_0}(\bfA_{i,\psi_t}[\p]/\Qcyc):=\op{ker} \left(H^1\left( \Q_{S}/\Qcyc, \bfA_{i,\psi_t}[\p]\right)\xrightarrow{\overline{\theta}_0}  \cH_p(\bfA_{i,\psi_t}[\p]/\Qcyc)\right)\]
\end{Definition}

\begin{Proposition}
\label{Kim 2.10}
For $i=1,2$, we have a natural isomorphism
\[
\Sel^{S_0}(\bfA_{i,\psi_t}[\p]/\Qcyc) \simeq \Sel^{S_0}(\bfA_{i,\psi_t}/\Qcyc)[\p].
\]
\end{Proposition}
\begin{proof}
We consider the diagram relating the two Selmer groups
\[
\begin{tikzcd}[column sep = small, row sep = large]
0\arrow{r} & \Sel^{S_0}(\bfA_{i,\psi_t}[\p]/\Qcyc) \arrow{r}  \arrow{d}{f} & H^1(\Q_{S}/\Qcyc, \bfA_{i,\psi_t}[\p]) \arrow{r} \arrow{d}{g} & \op{im} \overline{\theta}_0 \arrow{r} \arrow{d}{h} & 0\\
0\arrow{r} & \Sel^{S_0}(\bfA_{i,\psi_t}/\Qcyc) [\p] \arrow{r} & H^1(\Q_S/\Qcyc, \bfA_{i,\psi_t})[\p]  \arrow{r}  &\left(\op{im} \theta_0\right)[\p]\arrow{r}  & 0,
\end{tikzcd}\]
where the vertical maps are induced by the Kummer sequence. Additionally, we have
 \[H^0(\Q,\bfA_{i,\psi_t}[\p])=H^0(\Qcyc,\bfA_{i,\psi_t}[\p])^{\Gamma}=0.\]
Since $\Qcyc/\Q$ is a pro-$p$ extension, we deduce that $H^0(\Qcyc, \bfA_{i,\psi_t})=0$ and therefore $g$ is injective.
On the other hand, it clear that $g$ is surjective.

It only remains to show that $h$ is injective.
For $q=p$, denote by $\iota$ the natural map
\[\iota: \cH_q(\bfA_{i,\psi_t}[\p]/\Qcyc)\rightarrow \cH_q(\bfA_{i, \psi_t}/\Qcyc)[\p].\]
Consider the commutative square with injective horizontal maps
\[
\begin{tikzcd}[column sep = small, row sep = large]
 & \cH_q(\bfA_{i, \psi_t}[\p]/\Qcyc) \arrow{r}  \arrow{d}{\iota} & H^1\left( I_{\eta_p}, \bfA_{i, \psi_t}^-[\p]\right) \arrow{d}{j} \\
 & \cH_q(\bfA_{i, \psi_t}/\Qcyc)[\p]\arrow{r} &  H^1\left( I_{\eta_p}, \bfA_{i, \psi_t}^-\right)[\p].
\end{tikzcd}\]
Since $\bfA_{i, \psi}^-$ is unramified at $p$, it follows that $H^0(I_{\eta_p}, \bfA_{i, \psi}^-)=\bfA_{i, \psi}^-$ is divisible. It is easy to see that if $t\neq 0$, then, $H^0(I_{\eta_p}, \bfA_{i, \psi_t}^-)=0$.
The kernel of the map 
\[\iota: H^1(I_{\eta_p}, \bfA_{i, \psi_t}^-[\p])\rightarrow H^1(I_{\eta_p}, \bfA_{i, \psi_t}^-)[\p]\] is $H^0(I_{\eta_p}, \bfA_{i, \psi_t}^-)/\p=0$.
\end{proof}

\begin{Lemma}
The isomorphism $ \bfA_{1,\psi_t}[\p]\simeq  \bfA_{2,\psi_t}[\p]$ of Galois modules induces an isomorphism of residual Selmer groups 
\[
\Sel^{S_0}(\bfA_{1,\psi_t}[\p]/\Qcyc)\simeq \Sel^{S_0}(\bfA_{2,\psi_t}[\p]/\Qcyc).
\]
\end{Lemma}
\begin{proof}
Note that the $\op{G}_{p}$-action on $\bfA_{i, \psi_t}^+[\p]$ is ramified and that on $\bfA_{i, \psi_t}^-[\p]$ is via an unramified character. Let $\Phi:\bfA_{1,\psi_t}[\p]\xrightarrow{\sim} \bfA_{2,\psi_t}[\p]$ be a choice of isomorphism of Galois modules, in which case it is easy to see that $\Phi$ induces an isomorphism 
\[\Phi:\bfA_{1,\psi_t}^+[\p]\xrightarrow{\sim} \bfA_{2,\psi_t}^+[\p].\] As a result, we have an isomorphism of $\op{G}_{p}$-modules $\bfA_{1,\psi_t}^-[\p]\simeq \bfA_{2,\psi_t}^-[\p]$.
Clearly, $\Phi$ induces an isomorphism $H^1(\Q_S/\Qcyc, \bfA_{1,\psi_t}[\p])\xrightarrow{\sim} H^1(\Q_S/\Qcyc, \bfA_{2,\psi_t}[\p])$.
It suffices to show that for $q\in S$, the isomorphism $\Phi:\bfA_{1,\psi_t}[\p]\xrightarrow{\sim} \bfA_{2,\psi_t}[\p]$ induces an isomorphism 
\[
\mathcal{H}_q(\Qcyc,\bfA_{1,\psi_t}[\p])\xrightarrow{\sim} \mathcal{H}_q(\Qcyc,\bfA_{2,\psi_t}[\p]).
\]
This is clear for $q\neq p$.
For $q=p$, this follows from the fact that $\Phi$ induces an isomorphism $\bfA_{1,\psi_t}^-[\p]\xrightarrow{\sim} \bfA_{2,\psi_t}^-[\p]$.
\end{proof} 

\begin{Corollary}
\label{p-torsion of Sigma_0 fine selmer are iso}
The isomorphism $\bfA_{1,\psi_t}[\p]\simeq \bfA_{2,\psi_t}[\p]$ of Galois modules induces an isomorphism \[\Sel^{S_0}(\bfA_{1,\psi_t}/\Qcyc)[\p]\simeq \Sel^{S_0}(\bfA_{2,\psi_t}/\Qcyc)[\p].\]
\end{Corollary}

\begin{Lemma}
\label{lemma: the two mus are the same}
The $\mu$-invariant of the Selmer group $\Sel(\bfA_{i, \psi_t}/\Qcyc)$ coincides with that of $\Sel^{S_0}(\bfA_{i, \psi_t}/\Qcyc)$, i.e.,
\[\mu(\bfA_{i, \psi_t}/\Qcyc)=\mu^{S_0}(\bfA_{i, \psi_t}/\Qcyc).\]
\end{Lemma}
\begin{proof}
The result follows from Lemma \ref{local mu is 0}, which states that $\mathcal{H}_q(\bfA_{i, \psi_t}/\Qcyc)$ has $\mu=0$ for $q\neq p$.
\end{proof}
\begin{Proposition}\label{prop 3.10}Let the notation be as above. Then, we have that 
\[\mu(\bfA_{1,\psi_t}/\Qcyc)=0\Leftrightarrow \mu(\bfA_{2,\psi_t}/\Qcyc)=0.\]Moreover, if these $\mu$-invariants are $0$, then the imprimitive $\lambda$-invariants coincide, i.e.,
\[\lambda^{S_0}(\bfA_{1,\psi_t}/\Qcyc)=\lambda^{S_0}(\bfA_{2,\psi_t}/\Qcyc).\] Furthermore, we have \[\lambda(\bfA_{2,\psi_t}/\Qcyc)-\lambda(\bfA_{1,\psi_t}/\Qcyc)=\sum_{q\in S_0} \left(\sigma_q^{(1)}-\sigma_q^{(2)}\right),\]
where $\sigma_q^{(i)}$ is the $\Z_p$-corank of $\cH_q(\bfA_{i, \psi_t}/\Qcyc)$. 
\end{Proposition}
\begin{proof}
Set $\mathfrak{X}_i$ to denote $\Sel^{S_0}(\bfA_{i, \psi_t}/\Qcyc)$.
Lemma \ref{lemma: the two mus are the same} asserts that the $\mu$-invariant of $\mathfrak{X}_i$ coincides with $\mu(\bfA_{i, \psi_t}/\Qcyc)$.
Therefore, $\mu(\bfA_{i, \psi_t}/\Qcyc)=0$ if and only if $\mathfrak{X}_i$ is cofinitely generated as a $\Z_p$-module.
Note that $\mathfrak{X}_i$ is cofinitely generated as a $\Z_p$-module if and only if $\mathfrak{X}_i[\p]$ has finite cardinality.
Corollary \ref{p-torsion of Sigma_0 fine selmer are iso} asserts that $\mathfrak{X}_1[\p]\simeq \mathfrak{X}_2[\p]$; thus,
\[\mu(\bfA_{1,\psi_t}/\Qcyc)=0\Leftrightarrow \mu(\bfA_{2,\psi_t}/\Qcyc)=0.\]
\par Assume that $\mathfrak{X}_1[\p]$ (or equivalently $\mathfrak{X}_2[\p]$) is finite.
It follows from \cite[Proposition 2.5]{GV00} that $\mathfrak{X}_i$ has no proper $\Lambda$-submodules of finite index.
It is an easy exercise to show that $\mathfrak{X}_i$ therefore is a cofree $\cO$-module.
Therefore, $\mathfrak{X}_i\simeq (K/\cO)^{\lambda(\mathfrak{X}_i)}$, where $\lambda(\mathfrak{X}_i):=\frac{\lambda^{S_0}(\bfA_{i, \psi_t}/\Qcyc)}{[K:\Q_p]}$.
As a result,
\[\lambda(\mathfrak{X}_i)=\op{dim}_{\F_p} \mathfrak{X}_i[\p],\] and the isomorphism $\mathfrak{X}_1[\p]\simeq \mathfrak{X}_2[\p]$ implies that $\lambda(\mathfrak{X}_1)=\lambda(\mathfrak{X}_2)$. As a result, it follows that
\[\lambda^{S_0}(\bfA_{1,\psi_t}/\Qcyc)=\lambda^{S_0}(\bfA_{2,\psi_t}/\Qcyc).\]
The result then follows from (\ref{relating im primitive and classical lambda invariant}).
\end{proof}

\par Let $g_1$ and $g_2$ be $\p$-congruent Hecke eigencuspforms satisfying the conditions stated in the introduction. We denote by $\mu^{\op{alg}}_{S_0}(r_{g_i}\otimes \psi_t)$ and $\lambda^{\op{alg}}_{S_0}(r_{g_i}\otimes \psi_t)$ to denote the Iwasawa invariants of the imprimitive Selmer group $\op{Sel}_{p^\infty}^{S_0}(\bfA_{i,\psi_t}/\Q_{\op{cyc}})$ obtained by dropping conditions at the primes $q\in S_0$. Denote by $\mu^{\op{an}}_{S_0}(r_{g_i}\otimes \psi_t)$ and $\lambda^{\op{an}}_{S_0}(r_{g_i}\otimes \psi_t)$ the Iwasawa invariants of the $S_0$-imprimitive $p$-adic L-function $\mathscr{L}^{\op{an}}_{S_0}(r_{g_i}\otimes \psi_t)$ obtained by dropping Euler factors at primes $q\in S_0$.

\begin{Proposition}\label{boring prop}
Let $g_1$ and $g_2$ be as above and $\ast\in \{\op{an}, \op{alg}\}$. Then, $\mu^{\ast}(r_{g_i}\otimes \psi_t)=\mu^{\ast}_{S_0}(r_{g_i}\otimes \psi_t)$ for $i=1,2$ and 
\[\lambda^{\ast}_{S_0}(r_{g_i}\otimes \psi_t)=\lambda^{\ast}(r_{g_i}\otimes \psi_t)+\sum_{q\in S_0} \sigma_i^{(q)}.\]
\end{Proposition}

\begin{proof}
When $\ast=\op{alg}$, the result follows from Lemma \ref{local mu is 0} and \eqref{relating im primitive and classical lambda invariant}. On the other hand, when $\ast=\op{an}$, the result follows from (\ref{prim-imprim}). The equality of the local factors $\sigma_i(q)$ 
comes from the local Langlands correspondence of \cite{GJ78} for $GL(3)$. \end{proof}

\begin{Th}
Let $g_1$ and $g_2$ satisfy the conditions stated in the introduction. Suppose the relations $\mu^{\op{alg}}(r_{g_1}\otimes \psi_t)=\mu^{\op{an}}(r_{g_1}\otimes \psi_t)=0$ and $\lambda^{\op{alg}}(r_{g_1}\otimes \psi_t)=\lambda^{\op{an}}(r_{g_1}\otimes \psi_t)$ hold. Then, we have further equalities $\mu^{\op{alg}}(r_{g_2}\otimes \psi_t)=\mu^{\op{an}}(r_{g_2}\otimes \psi_t)=0$ and $\lambda^{\op{alg}}(r_{g_2}\otimes \psi_t)=\lambda^{\op{an}}(r_{g_2}\otimes \psi_t)$.
\end{Th}

\begin{proof}
Let $\ast\in \{\op{an}, \op{alg}\}$, it follows from Proposition \ref{boring prop} that the equalities $\mu^{\op{alg}}(r_{g_i}\otimes \psi_t)=\mu^{\op{an}}(r_{g_i}\otimes \psi_t)=0$ and $\lambda^{\op{alg}}(r_{g_i}\otimes \psi_t)=\lambda^{\op{an}}(r_{g_i}\otimes \psi_t)$ hold if and only if the relations $\mu_{S_0}^{\op{alg}}(r_{g_i}\otimes \psi_t)=\mu_{S_0}^{\op{an}}(r_{g_i}\otimes \psi_t)=0$ and $\lambda_{S_0}^{\op{alg}}(r_{g_i}\otimes \psi_t)=\lambda_{S_0}^{\op{an}}(r_{g_i}\otimes \psi_t)$ hold. Therefore, by assumption, these relations hold for $i=1$, and we are to deduce them for $i=2$. Proposition \ref{p-adic LFs congruent} asserts that there is a $p$-adic unit $u$ such that 
$\mathscr{L}^{\op{an}}_{S_0}(r_{g_1}\otimes \psi_t)\equiv u \mathscr{L}^{\op{an}}_{S_0}(r_{g_2}\otimes \psi_t)\mod{\p}$.

From the above congruence, we find that \[\mu^{\op{an}}_{S_0}(r_{g_1}\otimes \psi_t)=0\Leftrightarrow \mu^{\op{an}}_{S_0}(r_{g_2}\otimes \psi_t)=0\] and if these $\mu$-invariants vanish, then, \[\lambda^{\op{an}}_{S_0}(r_{g_1}\otimes \psi_t)= \lambda^{\op{an}}_{S_0}(r_{g_2}\otimes \psi_t).\]On the other hand, the same assertion for algebraic L-functions holds by Proposition \ref{prop 3.10}. Therefore, the result follows.
\end{proof}

\subsection{Data availability}  No data was generated or analysed in this paper. 

\subsection{Compliance with Ethical Standards}  The authors declare that they have no conflict of interest. 


\bibliographystyle{abbrv}
\bibliography{references}

\end{document}